\documentclass[times,sort&compress,3p]{elsarticle}
\journal{Journal of Multivariate Analysis}
\usepackage{amsmath, amssymb, enumerate, amsthm}

\RequirePackage[OT1]{fontenc}
\RequirePackage{amsthm,amsmath}
\RequirePackage[colorlinks,citecolor=blue,urlcolor=blue]{hyperref}
\RequirePackage[numbers]{natbib}

\makeatletter
\providecommand{\doi}[1]{%
  \begingroup
    \let\bibinfo\@secondoftwo
    \urlstyle{rm}%
    \href{http://dx.doi.org/#1}{%
      doi:\discretionary{}{}{}%
      \nolinkurl{#1}%
    }%
  \endgroup
}
\makeatother

\usepackage{epsfig}
\usepackage{graphicx}
\usepackage[font=small,labelfont=bf,justification=justified,singlelinecheck=false]{caption}
\usepackage{subcaption}
\usepackage{amsfonts,amssymb}
\usepackage{tikz}
\usepackage{xcolor}

\newtheorem{theorem}{Theorem}

\newtheorem{corollary}{Corollary}
\newtheorem{conjecture}{Conjecture}

\newtheorem{remark}{Remark}

\newcommand{\thb}{\bar{\theta}}
\newcommand{\tht}{\theta}
\newcommand{\Tht}{\mathbf{\Theta}}
\newcommand{\thh}{\hat{\theta}}
\newcommand{\Thh}{\mathbf{\hat{\Theta}}}
\newcommand{\ut}{u}
\newcommand{\uh}{\hat{u}}
\newcommand{\zt}{z}
\newcommand{\zh}{\hat{z}}
\newcommand{\ub}{\tilde{u}}
\newcommand{\Ub}{\mathbf{\tilde{U}}}
\newcommand{\Ut}{\mathbf{U}}
\newcommand{\Uh}{\mathbf{\hat{U}}}
\newcommand{\Zt}{\mathbf{Z}}
\newcommand{\Zh}{\mathbf{\hat{Z}}}
\newcommand{\R}{\mathbf{R}}
\newcommand{\E}{\mathbf{E}}
\newcommand{\Eta}{\mathbf{H}}
\newcommand{\TT}{\top}
\newcommand{\CT}{\mathsf{H}}
\newcommand{\Y}{\mathbf{Y}}
\newcommand{\Yn}{\mathbf{\tilde{Y}}}
\newcommand{\LP}{\mathbf{P}}
\newcommand{\X}{\mathbf{X}}
\newcommand{\Hess}{\mathbf{H}}
\newcommand{\Span}{\operatorname*{Span}}

\newcommand{\diag}{\operatorname*{diag}}
\newcommand{\bR}{\mathbb{R}}
\newcommand{\bC}{\mathbb{C}}
\newcommand{\bE}{{\rm E}}

\newcommand{\One}{\mathbf{1}}
\newcommand{\LLambda}{\mathbf{\Lambda}}
\newcommand{\VV}{\mathbf{V}}
\newcommand{\cN}{\mathcal{N}}
\newcommand{\cCN}{\mathcal{CN}}
\newcommand{\sigmab}{\bar{\sigma}}
\newcommand{\asto}{\overset{a.s.}{\longrightarrow}}

\usepackage{fancyhdr}
\fancypagestyle{license}{\fancyhf{}%
  \lfoot{%
  \footnotesize%
  \textcopyright\ 2018. Licensed under the CC-BY-NC-ND 4.0 license
  \url{http://creativecommons.org/licenses/by-nc-nd/4.0/}\\
  Published version available at \doi{10.1016/j.jmva.2018.06.002}
  }
  \rfoot{\thepage}
}

\usepackage{chngcntr}
\makeatletter
\newcommand*\arabicminus[2]{\expandafter\@arabicminus{\csname c@#1\endcsname}{#2}}
\newcommand*\@arabicminus[2]{\the\numexpr(#1)-#2\relax}
\makeatother

\makeatletter
\def\ps@pprintTitle{%
   \let\@oddhead\@empty
   \let\@evenhead\@empty
   \def\@oddfoot{\reset@font\hfil\thepage\hfil}
   \let\@evenfoot\@oddfoot
}
\makeatother

\begin{document}

\pdfbookmark[0]{Paper}{bookmark:body}

\begin{frontmatter}
\title{Asymptotic performance of PCA for high-dimensional heteroscedastic data}

\author[]{David Hong\corref{cor1}\fnref{t1}}
\ead{dahong@umich.edu}

\author{Laura Balzano\fnref{t2}}
\ead{girasole@umich.edu}

\author{Jeffrey A. Fessler\fnref{t3}}
\ead{fessler@umich.edu}

\address{
Department of Electrical Engineering and Computer Science\\
University of Michigan, Ann Arbor, MI 48109, USA
}

\cortext[cor1]{Corresponding author}
\fntext[t1]{This work was supported by the National Science Foundation Graduate Research Fellowship [DGE \#1256260].}
\fntext[t2]{This work was supported by the ARO [W911NF-14-1-0634]; and DARPA [DARPA-16-43-D3M-FP-037].}
\fntext[t3]{This work was supported by the UM-SJTU data science seed fund; and the NIH [U01 EB 018753].}

\begin{abstract}
Principal Component Analysis (PCA) is a classical method for reducing the dimensionality of data by projecting them onto a subspace that captures most of their variation.
Effective use of PCA in modern applications requires understanding its performance for data that are both high-dimensional and heteroscedastic.
This paper analyzes the statistical performance of PCA in this setting, i.e., for high-dimensional data drawn from a low-dimensional subspace and degraded by heteroscedastic noise.
We provide simplified expressions for the asymptotic PCA recovery of the underlying subspace, subspace amplitudes and subspace coefficients; the expressions enable both easy and efficient calculation and reasoning about the performance of PCA.
We exploit the structure of these expressions
to show that,
for a fixed average noise variance,
the asymptotic recovery of PCA
for heteroscedastic data
is always worse than that for homoscedastic data
(i.e., for noise variances that are equal
across samples).
Hence, while average noise variance is often a practically convenient measure for the overall quality of data, it gives an overly optimistic estimate of the performance of PCA for heteroscedastic data.
\end{abstract}


\end{frontmatter}

\thispagestyle{license}

\section{Introduction} \label{sct:intro}

Principal Component Analysis (PCA) is a classical method for reducing the dimensionality of data by representing them in terms of a new set of variables, called principal components, where variation in the data is largely captured by the first few principal components~\cite{jolliffe1986pca}.
This paper analyzes the asymptotic performance
of PCA
for data with heteroscedastic noise.
In particular, we consider the classical and commonly employed unweighted form of PCA that treats all samples equally and remains a natural choice in applications where estimates of the noise variances
are unavailable or one hopes the noise is ``close enough'' to being homoscedastic.
Our analysis uncovers several practical new insights for this setting; the findings both broaden our understanding of PCA and also precisely characterize the impact of heteroscedasticity.

Given zero-mean sample vectors $y_1,\dots,y_n \in \bC^d$, the first $k$ principal components~$\uh_1,\dots,\uh_k \in \bC^d$ and corresponding squared PCA amplitudes $\thh_1^2,\dots,\thh_k^2 \in \bR_+$ are the first $k$ eigenvectors and eigenvalues, respectively, of the sample covariance matrix $(y_1y_1^\CT + \cdots + y_ny_n^\CT)/n$.
The associated score vectors~$\zh^{(1)},\dots,\zh^{(k)}\in\bC^n$ are  standardized projections given, for each $i \in \{ 1, \ldots, k\}$, by $\zh^{(i)}=(1/\thh_i)\{\uh_i^\CT(y_1, \ldots, y_n)\}^\CT$.
The principal components~$\uh_1, \ldots, \uh_k$, PCA amplitudes~$\thh_1, \ldots, \thh_k$ and score vectors~$\zh^{(1)},\dots,\zh^{(k)}$ are efficiently obtained from the data matrix $(y_1 , \ldots ,y_n) \in \bC^{d \times n}$ as its left singular vectors, (scaled) singular values and (scaled) right singular vectors, respectively.

A natural setting for PCA is when data are noisy measurements of points drawn from a subspace. In this case, the first few principal components $\uh_1, \ldots, \uh_k$ form an estimated  basis for the underlying subspace; if they recover the underlying subspace accurately then the low-dimensional scores~$\zh^{(1)}, \ldots, \zh^{(k)}$ will largely capture the meaningful variation in the data.
This paper analyzes how well the first $k$ principal components~$\uh_1, \ldots, \uh_k$, PCA amplitudes~$\thh_1, \ldots, \thh_k$ and score vectors $\zh^{(1)}, \ldots, \zh^{(k)}$ recover their underlying counterparts when the data are heteroscedastic, that is, when the noise in the data has non-uniform variance across samples.

\subsection{High-dimensional, heteroscedastic data}

Dimensionality reduction is a fundamental task, so PCA has been applied in a broad variety of both traditional and modern settings.
See~\cite{jolliffe1986pca} for a thorough review of PCA and some of its important traditional applications.
A sample of modern application areas include medical imaging~\cite{ardekani1999adi,pedersen2009ktp}, classification for cancer data~\cite{sharma2015and}, anomaly detection on computer networks~\cite{lakhina2004dnw}, environmental monitoring~\cite{papadimitriou2005spd,wagner1996sdu} and genetics~\cite{leek2011acs}, to name just a few.

It is common in modern applications in particular for the data to be high-dimensional (i.e., the number of variables measured is comparable with or larger than the number of samples), which has motivated the development of new techniques and theory for this regime~\cite{johnstone2009sco}. It is also common for modern data sets to have heteroscedastic noise. For example,~Cochran and Horne \cite{cochran1977swp} apply a PCA variant to spectrophotometric data  from the study of chemical reaction kinetics. The spectrophotometric data are absorptions at various wavelengths over time, and measurements are averaged over increasing windows of time causing the amount of noise to vary across time. Another example is given in~\cite{tamuz2005cse}, where data are astronomical measurements of stars taken at various times; here changing atmospheric effects cause the amount of noise to vary across time. More generally, in the era of big data where inference is made using numerous data samples drawn from a myriad of different sources, one can expect that both high-dimensionality and heteroscedasticity will be the norm. It is important to understand the performance of PCA in such settings.

\subsection{Contributions of this paper} \label{sct:contribution}

This paper provides simplified expressions for the performance of PCA from heteroscedastic data in the limit as both the number of samples and dimension tend to infinity. The expressions quantify the asymptotic recovery of an underlying subspace, subspace amplitudes and coefficients by the principal components, PCA amplitudes and scores, respectively. The asymptotic recoveries are functions of the samples per ambient dimension, the underlying subspace amplitudes and the distribution of noise variances. Forming the expressions involves first connecting several results from random matrix theory~\cite{bai2010sao,benaych2012tsv} to obtain initial expressions for asymptotic recovery that are difficult to evaluate and analyze, and then exploiting a nontrivial structure in the expressions to obtain much simpler algebraic descriptions.
These descriptions enable both easy and efficient calculation and reasoning about the asymptotic performance of PCA.

The impact of heteroscedastic noise, in particular, is not immediately obvious given results of prior literature. How much do a few noisy samples degrade the performance of PCA?
Is heteroscedasticity
ever beneficial for PCA?
Our simplified expressions enable such questions to be answered.
In particular, we use these expressions to show that,
for a fixed average noise variance,
the asymptotic subspace recovery, amplitude recovery and coefficient recovery
are all worse for heteroscedastic data
than for homoscedastic data
(i.e., for noise variances that are equal
across samples), confirming a conjecture in~\cite{hong2016tat}.
Hence, while average noise variance is often a practically convenient measure for the overall quality of data, it gives an overly optimistic estimate of PCA performance.
This analysis provides a deeper understanding of how PCA performs in the presence of heteroscedastic noise.

\subsection{Relationship to previous works}

Homoscedastic noise has been well-studied, and there are many nice results characterizing PCA in this setting.
Benaych-Georges and Nadakuditi \cite{benaych2012tsv} give an expression for asymptotic subspace recovery, also found in~\cite{johnstone2009oca,nadler2008fsa,paul2007aos}, in the limit as both the number of samples and ambient dimension tend to infinity.
As argued in~\cite{johnstone2009oca}, the expression in~\cite{benaych2012tsv} reveals that asymptotic subspace recovery is perfect only when the number of samples per ambient dimension tends to infinity, so PCA is not (asymptotically) consistent for high-dimensional data.
Various alternatives~\cite{bickel2008crb,elkaroui2008onc,johnstone2009oca} can regain consistency by exploiting sparsity in the covariance matrix or in the principal components.
As discussed in~\cite{benaych2012tsv,nadler2008fsa}, the expression in~\cite{benaych2012tsv} also exhibits a phase transition: the number of samples per ambient dimension must be sufficiently high to obtain non-zero subspace recovery (i.e., for any subspace recovery to occur).
This paper generalizes the expression in~\cite{benaych2012tsv} to heteroscedastic noise; homoscedastic noise is a special case and is discussed in Section~\ref{sct:special}.
Once again, (asymptotic) consistency is obtained when the number of samples per ambient dimension tends to infinity, and there is a phase transition between zero recovery and non-zero recovery.

PCA is known to generally perform well in the presence of low to moderate homoscedastic noise and in the presence of missing data~\cite{chatterjee2015meb}. When the noise is standard normal, PCA gives the maximum likelihood (ML) estimate of the subspace~\cite{tipping1999ppc}.
In general,~\cite{tipping1999ppc} proposes finding the ML estimate via expectation maximization.
Conventional PCA is not an ML estimate of the subspace for heteroscedastic data, but it remains a natural choice in applications where we might expect noise to be heteroscedastic but hope it is ``close enough'' to being homoscedastic.
Even for mostly homoscedastic data, however, PCA performs poorly when the heteroscedasticity is due to gross errors (i.e., outliers)~\cite{devlin1981reo,huber1981rs,jolliffe1986pca}, which has motivated the development and analysis of robust variants; see~\cite{candes2011rpc,chandrasekaran2011rsi,croux2005hbe,he2012igo,he2014orb,lerman2015rco,qiu2014rpc,xu2012rpv,zhan2016oao} and their corresponding bibliographies.
This paper provides expressions for asymptotic recovery that enable rigorous understanding of the impact of heteroscedasticity.

The generalized spiked covariance model, proposed and analyzed in~\cite{bai2012ose} and~\cite{yao2015lsc}, generalizes homoscedastic noise in an alternate way.
It extends the Johnstone spiked covariance model~\cite{johnstone2001otd,johnstone2009oca} (a particular homoscedastic setting) by using a population covariance that allows, among other things, non-uniform noise variances {\em within} each sample.
Non-uniform noise variances within each sample may arise, for example, in applications where sample vectors are formed by concatenating the measurements of intrinsically different quantities. This paper considers data with non-uniform noise variances {\em across} samples instead; we model noise variances {\em within} each sample as uniform. Data with non-uniform noise variances across samples arise, for example, in applications where samples come from heterogeneous sources, some of which are better quality (i.e., lower noise) than others.
See Section~\ref{sct:spike} of the Online Supplement for a more detailed discussion of connections to spiked covariance models.

Our previous work~\cite{hong2016tat} analyzed the subspace recovery of PCA for heteroscedastic noise but was limited to real-valued data coming from a random one-dimensional subspace where the number of samples exceeded the data dimension.
This paper extends that analysis to the more general setting of real- or complex-valued data coming from a deterministic low-dimensional subspace where the number of samples no longer needs to exceed the data dimension.
This paper also extends the analysis of~\cite{hong2016tat} to include the recovery of the underlying subspace amplitudes and coefficients.
In both works, we use the main results of~\cite{benaych2012tsv} to obtain initial expressions  relating asymptotic recovery to the limiting noise singular value distribution.

The main results of~\cite{nadler2008fsa} provide non-asymptotic results (i.e., probabilistic approximation results for finite samples in finite dimension) for homoscedastic noise limited to the special case of one-dimensional subspaces.
Signal-dependent noise was recently considered in~\cite{vaswani2016cpp}, where they analyze the performance of PCA and propose a new generalization of PCA that performs better in certain regimes.
A recent extension of~\cite{benaych2012tsv} to linearly reduced data is presented in~\cite{dobriban2016pca} and may be useful for analyzing weighted variants of PCA.
Such analyses are beyond the scope of this paper, but are interesting avenues for further study.

\subsection{Organization of the paper} \label{sct:organization}
Section~\ref{sct:result}
describes the model we consider and states the main results: simplified expressions for asymptotic PCA recovery and the fact that PCA performance is best (for a fixed average noise variance) when the noise is homoscedastic.
Section~\ref{sct:analysis} uses the main results to provide a qualitative
analysis of how the model parameters (e.g., samples per ambient dimension and the distribution of noise variances) affect PCA performance under heteroscedastic noise.
Section~\ref{sct:experiment} compares the asymptotic recovery with non-asymptotic (i.e., finite) numerical simulations. The simulations demonstrate good
agreement as the ambient dimension and number of samples grow large; when
these values are small the asymptotic recovery and simulation differ but have the
same general behavior.
Sections~\ref{sct:proof} and~\ref{sct:bound} prove the main results.
Finally, Section~\ref{sct:discussion} discusses the findings and describes
avenues for future work.

\section{Main results} \label{sct:result}

\subsection{Model for heteroscedastic data} \label{sct:model}
We model $n$ heteroscedastic sample vectors
$y_{1},\ldots ,y_{n}\in \bC^{d}$
from a $k$-dimensional subspace as
\begin{equation}
y_{i}=\Ut\Tht\zt_i+\eta_{i}\varepsilon_{i}
= \sum_{j=1}^k \tht_j \ut_j (\zt^{(j)}_i)^* + \eta_i \varepsilon_i.
\label{eq:model}
\end{equation}
The following are deterministic:
\begin{itemize}
\item [] $\Ut=(\ut_1, \dots, \ut_k) \in \bC^{d\times k}$ forms an orthonormal basis for the subspace,
\item [] $\Tht=\diag(\tht_1,\dots,\tht_k) \in \bR_+^{k \times k}$ is a diagonal matrix of amplitudes,
\item [] $\eta_i \in \{\sigma_1,\dots,\sigma_L\}$ are each one of $L$ noise standard deviations $\sigma_1,\dots,\sigma_L$,
\end{itemize}
and we define $n_1$ to be the number of samples with $\eta_i=\sigma_1$, $n_2$ to be the number of samples with $\eta_i = \sigma_2$ and so on, where $n_1+\cdots+n_L=n$.

The following are random and independent:
\begin{itemize}
\item [] $\zt_i \in \bC^k$ are iid sample coefficient vectors that have iid entries with mean $\bE (\zt_{ij}) = 0$, variance $\bE|\zt_{ij}|^2 = 1$, and a distribution satisfying the log-Sobolev inequality~\cite{anderson2009ait},
\item [] $\varepsilon_i \in \bC^d$ are unitarily invariant iid noise vectors that have iid entries with mean $\bE (\varepsilon_{ij}) = 0$, variance $\bE|\varepsilon_{ij}|^2=1$ and bounded fourth moment $\bE|\varepsilon_{ij}|^4 < \infty$,
\end{itemize}
and we define the $k$ (component) coefficient vectors $\zt^{(1)},\dots,\zt^{(k)} \in \bC^n$ such that the $i$th entry of $\zt^{(j)}$ is $\zt^{(j)}_i=(\zt_{ij})^*$, the complex conjugate of the $j$th entry of $\zt_i$.
Defining the coefficient vectors in this way is convenient for stating and proving the results that follow, as they more naturally correspond to right singular vectors of the data matrix formed by concatenating $y_1,\dots,y_n$ as columns.

The model extends the Johnstone spiked covariance model~\cite{johnstone2001otd,johnstone2009oca}
by incorporating heteroscedasticity (see Section~\ref{sct:spike} of the Online Supplement for a detailed discussion).
We also allow complex-valued data, as it is of interest in important signal processing applications such as medical imaging; for example, data obtained in magnetic resonance imaging are complex-valued.

\begin{remark} \label{rm:unitaryinv}
\em
By unitarily invariant, we mean that left multiplication of~$\varepsilon_i$ by any unitary matrix does not change the joint distribution of its entries.
As in our previous work~\cite{hong2016tat}, this assumption can be dropped if instead the subspace~$\Ut$ is randomly drawn according to either the ``orthonormalized model'' or  ``iid model'' of~\cite{benaych2012tsv}. Under these models, the subspace~$\Ut$ is randomly chosen in an isotropic manner.
\end{remark}

\begin{remark} \label{rm:gaussian}
\em
The above conditions are satisfied, for example, when the entries $\zt_{ij}$ and $\varepsilon_{ij}$ are circularly symmetric complex normal $\cCN(0,1)$ or real-valued normal $\cN(0,1)$.
Rademacher random variables (i.e., $\pm 1$ with equal probability) are another choice for coefficient entries $z_{ij}$;
see Section 2.3.2 of~\cite{anderson2009ait} for discussion of the log-Sobolev inequality.
We are unaware of non-Gaussian distributions satisfying all conditions for noise entries $\varepsilon_{ij}$, but as noted in Remark~\ref{rm:unitaryinv}, unitary invariance can be dropped if we assume the subspace is randomly drawn as in~\cite{benaych2012tsv}.
\end{remark}

\begin{remark}
\em
The assumption that noise entries $\varepsilon_{ij}$ are identically distributed with bounded fourth moment can be relaxed when they are real-valued as long as an aggregate of their tails still decays sufficiently quickly, i.e., as long as they satisfy Condition~1.3 from~\cite{pan2010sco}.
In this setting, the results of~\cite{pan2010sco} replace those of~\cite{bai2010sao} in the proof.
\end{remark}

\subsection{Simplified expressions for asymptotic recovery} \label{sct:rs}
The following theorem describes how well the PCA estimates $\uh_1, \ldots, \uh_k$, $\thh_1, \ldots, \thh_k$ and $\zh^{(1)}, \ldots, \zh^{(k)}$ recover the underlying subspace basis $\ut_1, \ldots, \ut_k$, subspace amplitudes $\tht_1, \ldots, \tht_k$ and coefficient vectors $\zt^{(1)}, \ldots, \zt^{(k)}$, as a function of the sample-to-dimension ratio $n/d \to c >0$, the subspace amplitudes $\tht_1, \ldots, \tht_k$, the noise variances $\sigma_1^2, \ldots, \sigma_L^2$ and corresponding proportions $n_\ell/n \to p_\ell$ for each $\ell \in \{ 1, \ldots, L\}$.
One may generally expect performance to improve with increasing sample-to-dimension ratio and subspace amplitudes; Theorem~\ref{thm:rs} provides the precise dependence on these parameters as well as on the noise variances and their proportions.

\begin{theorem}[Recovery of individual components] \label{thm:rs}
Suppose that the sample-to-dimension ratio $n/d \to c > 0$ and the noise variance proportions $n_\ell/n \to p_\ell$ for $\ell \in \{1,\dots,L\}$ as $n,d \to \infty$.
Then the $i$th PCA amplitude $\thh_i$ is such that
\begin{equation} \label{eq:rstheta}
\thh_i^2 \asto
\frac{1}{c}\max( \alpha,\beta_i) \left\{ 1+c\sum_{\ell=1}^L\frac{p_\ell\sigma_\ell^2}{\max (\alpha,\beta_i) -\sigma_\ell^2}\right\},
\end{equation}
where $\alpha$ and $\beta_i$ are, respectively, the largest real roots of
\begin{align}
A(x) = 1-c \sum_{\ell=1}^L \frac{p_\ell \sigma_\ell^4}{(x-\sigma_\ell^2)^2}, \quad
B_i(x) = 1-c\tht_i^2 \sum_{\ell=1}^L
\frac{p_\ell}{x-\sigma_\ell^2}.
\label{eq:ABdef}
\end{align}
Furthermore, if $A(\beta_i) > 0$, then the $i$th principal component $\uh_i$ is such that
\begin{align}
\label{eq:rs}
\left|\langle \uh_i,\Span\{\ut_j:\tht_j  =  \tht_i\} \rangle\right|^2
\asto \frac{A(\beta_i)}{\beta_i B_i'(\beta_i)} , \quad
\left|\langle \uh_i,\Span\{\ut_j:\tht_j\neq \tht_i\} \rangle\right|^2
\asto 0 ,
\end{align}
the normalized score vector $\zh^{(i)}/\sqrt{n}$ is such that
\begin{align}
\left|\left\langle \frac{\zh^{(i)}}{\sqrt{n}},\Span\{\zt^{(j)}:\tht_j  =  \tht_i\} \right\rangle\right|^2
&\asto
\frac{A(\beta_i)}{c\{\beta_i+(1-c)\tht_i^2\} B_i'(\beta_i)}, \label{eq:rsright} \\
\left|\left\langle \frac{\zh^{(i)}}{\sqrt{n}},\Span\{\zt^{(j)}:\tht_j\neq \tht_i\} \right\rangle\right|^2
&\asto 0,\nonumber
\end{align}
and
\begin{equation}
\sum_{j:\tht_j=\tht_i}
\langle \uh_i,\ut_j \rangle \left
\langle \frac{\zh^{(i)}}{\sqrt{n}},\frac{\zt^{(j)}}{\|\zt^{(j)}\|} \right\rangle^*
\asto
\frac{A(\beta_i)}{\sqrt{c\beta_i \{ \beta_i+(1-c)\tht_i^2\} } B_i'(\beta_i)} \label{eq:rsmix}
. \end{equation}
\end{theorem}

\noindent Section~\ref{sct:proof} presents the proof of Theorem~\ref{thm:rs}.
The expressions can be easily and efficiently computed. The hardest part is finding the largest roots of the univariate rational functions $A(x)$ and $B_i(x)$, but off-the-shelf solvers can do this efficiently. See~\cite{hong2016tat} for an example of similar calculations.

The projection $\vert\langle \uh_{i},\Span\{
\ut_{j}:\tht_{j}=\tht_{i}\} \rangle \vert ^{2}$
in Theorem~\ref{thm:rs} is the square cosine principal angle between the $i$th principal
component $\uh_{i}$ and the span of the basis elements with
subspace amplitudes equal to $\tht_{i}$. When the subspace amplitudes are
distinct, $\vert \langle \uh_{i},\Span\{ \ut_{j}:\tht_{j}=\tht_{i}\} \rangle
\vert ^{2}=\left\vert \left\langle \uh_{i},\ut_{i}\right\rangle \right\vert ^{2}$ is the square cosine angle between $\uh_{i}$ and $\ut_{i}$. This value is related by a constant to the squared error between the two (unit norm) vectors and is one among several natural performance metrics for subspace estimation.
Similar observations hold for $|\langle \zh^{(i)}/\sqrt{n},\Span\{\zt^{(j)}:\tht_j = \tht_i\} \rangle|^2$. Note that $\zh^{(i)}/\sqrt{n}$ has unit norm.

The expressions~\eqref{eq:rs},~\eqref{eq:rsright} and~\eqref{eq:rsmix}
apply only if $A(\beta_i)>0$.
The following conjecture predicts a phase transition at $A(\beta_i)=0$ so that asymptotic recovery is zero for $A(\beta_i) \leq 0$.

\begin{conjecture}[Phase transition] \label{conj:rs}
Suppose (as in Theorem~\ref{thm:rs}) that the sample-to-dimension ratio $n/d \to c > 0$ and the noise variance proportions $n_\ell/n \to p_\ell$  for $\ell \in \{1,\dots,L\}$ as $n,d \to \infty$.
If $A(\beta_i) \leq 0$, then the $i$th principal component $\uh_{i}$ and the normalized score vector $\zh^{(i)}/\sqrt{n}$ are such that
\begin{align*}
|\langle \uh_i,\Span\{\ut_1,\dots,\ut_k\} \rangle|^2
\asto 0, \quad
\left|\left\langle \frac{\zh^{(i)}}{\sqrt{n}},\Span\{\zt^{(1)},\dots,\zt^{(k)}\} \right\rangle\right|^2
\asto 0 \nonumber.
\end{align*}
\end{conjecture}

\noindent This conjecture is true for a data model having Gaussian coefficients and homoscedastic Gaussian noise as shown in~\cite{paul2007aos}.
It is also true for a one-dimensional subspace (i.e., $k=1$) as we showed in~\cite{hong2016tat}.
Proving it in general would involve showing that the singular values of the matrix whose columns are the noise vectors exhibit repulsion behavior; see Remark~2.13 of~\cite{benaych2012tsv}.

\subsection{Homoscedastic noise as a special case} \label{sct:special}
For homoscedastic noise with variance $\sigma^2$,
$A(x) = 1-c\sigma^4/(x-\sigma^2)^2$ and $B_i(x) = 1-c\tht_i^2/(x-\sigma^2)$. The largest real roots of these functions are, respectively,
$\alpha=(1+\sqrt{c})\sigma^2$ and $\beta_i=\sigma^2+c\tht_i^2$.
Thus the asymptotic PCA amplitude~\eqref{eq:rstheta} becomes
\begin{equation} \label{eq:rsthetah}
\thh_i^2 \asto
\begin{cases}
\tht_i^2 \{ 1+\sigma^2/(c\tht_i^2) \} (1+ \sigma^2/\tht_i^2 ) & \text{if } c\tht_{i}^{4} > \sigma^4, \\
\sigma^2 (1+ 1/\sqrt{c} )^2
& \text{otherwise}.
\end{cases}
\end{equation}
Further, if $c\tht_{i}^{4} > \sigma^4$, then the non-zero portions of asymptotic subspace recovery~\eqref{eq:rs} and coefficient recovery~\eqref{eq:rsright} simplify to
\begin{align}
\left|\langle \uh_i,\Span\{\ut_j:\tht_j  =  \tht_i\} \rangle\right|^2
&\asto \frac{c-\sigma^4/\tht_i^4}{c+\sigma^2/\tht_i^2}, \label{eq:rsh} \\
\left|\left\langle \frac{\zh^{(i)}}{\sqrt{n}},\Span\{\zt^{(j)}:\tht_j  =  \tht_i\} \right\rangle\right|^2
&\asto
\frac{c-\sigma^4/\tht_i^4}{c(1+\sigma^2/\tht_i^2)} \nonumber,
\end{align}
These limits agree with the homoscedastic results in~\cite{benaych2012tsv,biehl1994smo,johnstone2009oca,nadler2008fsa,paul2007aos}. As noted in Section~\ref{sct:rs}, Conjecture~\ref{conj:rs} is known to be true when the coefficients are Gaussian and the noise is both homoscedastic and Gaussian, in which case~\eqref{eq:rsh} becomes
\begin{align*}
\left|\langle \uh_i,\Span\{\ut_j:\tht_j  =  \tht_i\} \rangle\right|^2
&\asto \max \left(0,\frac{c-\sigma^4/\tht_i^4}{c+\sigma^2/\tht_i^2}\right) ,  \\
\left|\left\langle \frac{\zh^{(i)}}{\sqrt{n}},\Span\{\zt^{(j)}:\tht_j  =  \tht_i\} \right\rangle\right|^2
&\asto
\max\left\{0,\frac{c-\sigma^4/\tht_i^4}{c(1+\sigma^2/\tht_i^2)}\right\} \nonumber.
\end{align*}
See Section~2 of~\cite{johnstone2009oca} and Section~2.3 of~\cite{paul2007aos} for a discussion of this result.

\subsection{Bias of the PCA amplitudes} \label{sct:ampbias}
The simplified expression in~\eqref{eq:rstheta} enables us to immediately make two observations about the recovery of the subspace amplitudes~$\tht_1,\dots,\tht_k$ by the PCA amplitudes~$\thh_1,\dots,\thh_k$.

\begin{remark}[Positive bias in PCA amplitudes] \label{rm:bias}
\em
The largest real root $\beta_i$ of $B_i(x)$ is greater than $\max_\ell (\sigma_\ell^2)$. Thus $1/(\beta_i-\sigma_\ell^2) > 1/\beta_i$ for $\ell \in \{1,\dots,L\}$ and so evaluating~\eqref{eq:ABdef} at $\beta_i$ yields
\begin{equation*}
0=B_i(\beta_i)
=1-c\tht_i^2 \sum_{\ell=1}^L \frac{p_\ell}{\beta_i-\sigma_\ell^2}
<1-c\tht_i^2 \frac{1}{\beta_i}.
\end{equation*}
As a result, $\beta_i >c\tht_i^2$, so the asymptotic PCA amplitude~\eqref{eq:rstheta} exceeds the subspace amplitude, i.e., $\thh_i$ is positively biased and is thus an inconsistent estimate of $\tht_i$.
This is a general phenomenon for noisy data and motivates asymptotically optimal shrinkage in~\cite{nadakuditi2014oaa}.
\end{remark}

\begin{remark}[Alternate formula for amplitude bias] \label{cor:rstheta}
\em
If $A(\beta_i) \geq 0$, then $\beta_i \geq \alpha$ because $A(x)$ and $B_i(x)$ are both increasing functions for $x > \max_\ell(\sigma_\ell^2)$. Thus, the asymptotic amplitude bias is
\begin{align} \label{eq:rstheta_alt}
\frac{\thh_i^2}{\tht_i^2} \asto
\frac{\beta_i}{c\tht_i^2}\left(1+c\sum_{\ell=1}^L\frac{p_\ell\sigma_\ell^2}{\beta_i-\sigma_\ell^2}\right)
&=\frac{\beta_i}{c\tht_i^2}\left\{ 1+c\sum_{\ell=1}^Lp_\ell\left(-1+\frac{\beta_i}{\beta_i-\sigma_\ell^2}\right)\right\} \nonumber \\
&=\frac{\beta_i}{c\tht_i^2}\left(1+\beta_ic\sum_{\ell=1}^L\frac{p_\ell}{\beta_i-\sigma_\ell^2}-c\right)
=\frac{\beta_i}{c\tht_i^2} \left[1+\frac{\beta_i}{\tht_i^2}\{1-B_i(\beta_i)\}-c\right] \nonumber\\
&=\frac{\beta_i}{c\tht_i^2}\left(1+\frac{\beta_i}{\tht_i^2}-c\right)
=1+\left(\frac{\beta_i}{c\tht_i^2}-1\right)\left(\frac{\beta_i}{\tht_i^2}+1\right),
\end{align}
where we have applied~\eqref{eq:rstheta}, divided the summand with respect to $\sigma_\ell^2$, used the facts that $p_1+\cdots+p_L=1$ and  $B_i(\beta_i)=0$, and finally factored.
The expression~\eqref{eq:rstheta_alt} shows that the positive bias is an increasing function of $\beta_i$ when $A(\beta_i) \geq 0$.
\end{remark}

\subsection{Overall subspace  and signal recovery}
Overall subspace recovery is more useful than individual component recovery when subspace amplitudes are equal and so individual basis elements are not identifiable.
It is also more relevant when we are most interested in recovering or denoising low-dimensional signals in a subspace.
Overall recovery of the low-dimensional signal, quantified here by mean square error, is useful for understanding  how well PCA ``denoises'' the data taken as a whole.

\begin{corollary}[Overall recovery]\label{cor:rs}
Suppose (as in Theorem~\ref{thm:rs}) that the sample-to-dimension ratio $n/d \to c > 0$ and the noise variance proportions $n_\ell/n \to p_\ell$  for $\ell \in \{1,\dots,L\}$ as $n,d \to \infty$.
If $A(\beta_{1}) ,\dots ,A(\beta_{k})>0$, then the subspace estimate $\Uh= (\uh_1 , \dots , \uh_k) \in \bC^{d \times k}$ from PCA is such that
\begin{equation} \label{eq:cor_rs}
\frac{1}{k}\, \| \Uh^\CT \Ut\|_F^2
\asto \frac{1}{k}\sum_{i=1}^{k}\frac{A\left(
\beta_{i}\right) }{\beta_{i}B_{i}^{\prime }\left( \beta_{i}\right) },
\end{equation}
and the mean square error is
\begin{align} \label{eq:mse}
\frac{1}{n} \sum_{i=1}^n &\left\|\Ut\Tht\zt_i - \Uh\Thh\zh_i\right\|_{2}^2 \asto
\sum_{i=1}^k 2\left\{\tht_i^2 - \frac{A(\beta_i)}{cB'(\beta_i)}\right\} +\left(\frac{\beta_i}{c\tht_i^2}-1\right)(\beta_i+\tht_i^2 ),
\end{align}
where $A(x)$, $B_{i}(x) $ and $\beta_{i}$ are as in Theorem~\ref{thm:rs},
and $\zh_i$ is the vector of score entries for the $i$th sample.
\end{corollary}

\begin{proof}[Proof of Corollary~\ref{cor:rs}]
The subspace recovery can be decomposed as
\begin{equation*}
\frac{1}{k}\, \| \Uh^\CT \Ut\|_F^2
=\frac{1}{k}\, \sum_{i=1}^{k}\left\Vert \uh_{i}^{\CT }\Ut_{j:\tht_{j}=\tht_{i}}\right\Vert_{2}^{2}+\left\Vert \uh_{i}^{\CT }\Ut_{j:\tht_{j}\neq \tht_{i}}\right\Vert_{2}^{2},
\end{equation*}
where the columns of $\Ut_{j:\tht_{j}=\tht_{i}}$ are
the basis elements $\ut_{j}$ with subspace amplitude~$\tht_j$ equal to $\tht_{i}$,
and the remaining basis elements are the columns of $\Ut_{j:\tht_{j}\neq\tht_{i}}$.
Asymptotic overall subspace recovery~\eqref{eq:cor_rs} follows by noting that these terms are exactly the square cosine principal angles in~\eqref{eq:rs} of  Theorem~\ref{thm:rs}.

The mean square error can also be decomposed as
\begin{align}
\frac{1}{n} \sum_{i=1}^n \left\|\Ut\Tht\zt_i - \Uh\Thh\zh_i\right\|_{2}^2 &= \left\|\Ut\Tht\left(\frac{1}{\sqrt{n}}\Zt\right)^\CT - \Uh\Thh\left(\frac{1}{\sqrt{n}}\Zh\right)^\CT\right\|_F^2 \nonumber \\
&= \sum_{i=1}^k \tht_i^2\left[\left\|\frac{\zt^{(i)}}{\sqrt{n}}\right\|_{2}^2 + \frac{\thh_i^2}{\tht_i^2}
-2\Re \left\{ \frac{\thh_i}{\tht_i}\sum_{j=1}^k
\frac{\tht_j}{\tht_i}
\langle \uh_i,\ut_j \rangle \left
\langle \frac{\zh^{(i)}}{\sqrt{n}},\frac{\zt^{(j)}}{\sqrt{n}} \right\rangle^*\right\}\right], \label{eq:mse_decomp}
\end{align}
where
$\Zt = (\zt^{(1)},\ldots,\zt^{(k)}) \in \bC^{n\times k}$,
$\Zh = (\zh^{(1)},\ldots,\zh^{(k)}) \in \bC^{n\times k}$
and $\Re$ denotes the real part of its argument.
The first term of~\eqref{eq:mse_decomp} has almost sure limit $1$ by the law of large numbers. The almost sure limit of the second term is obtained from~\eqref{eq:rstheta_alt}.
We can disregard the summands in the inner sum for which $\tht_j \neq \tht_i$; by~\eqref{eq:rs} and~\eqref{eq:rsright} these terms have an almost sure limit of zero (the inner products both vanish). The rest of the inner sum
\begin{equation*}
\sum_{j:\tht_j=\tht_i}
\frac{\tht_j}{\tht_i}
\langle \uh_i,\ut_j \rangle \left
\langle \frac{\zh^{(i)}}{\sqrt{n}},\frac{\zt^{(j)}}{\sqrt{n}} \right\rangle^* =
\sum_{j:\tht_j=\tht_i} (1)
\langle \uh_i,\ut_j \rangle \left
\langle \frac{\zh^{(i)}}{\sqrt{n}},\frac{\zt^{(j)}}{\sqrt{n}} \right\rangle^*
\end{equation*}
has the same almost sure limit as in~\eqref{eq:rsmix} because $\|\zt^{(i)}/\sqrt{n}\|^2 \to 1$ as $n \to \infty$.
Combining these almost sure limits and simplifying yields~\eqref{eq:mse}.
\end{proof}

\subsection{Importance of homoscedasticity} \label{sct:optim}
How important is homoscedasticity for PCA? Does having some low noise data outweigh the cost of introducing heteroscedasticity?
Consider the following three settings:
\begin{itemize}
\item [1.-] All samples have noise variance $1$ (i.e., data are homoscedastic).
\item [2.-] $99\%$ of samples have noise variance $1.01$ but $1\%$ have noise variance $0.01$.
\item [3.-] $99\%$ of samples have noise variance $0.01$ but $1\%$ have noise variance $99.01$.
\end{itemize}
In all three settings, the average noise variance is $1$.
We might expect PCA to perform well in Setting~1 because it has the smallest maximum noise variance.
However, Setting~2 may seem favorable because we obtain samples with very small noise,
and suffer only a slight increase in noise for the rest.
Setting~3 may seem favorable because most of the samples have very small noise. However, we might also expect PCA to perform poorly because $1\%$ of samples have very large noise and will likely produce gross errors (i.e., outliers).
Between all three, it is not initially obvious what setting PCA will perform best in.
The following theorem shows that PCA performs best when the noise is homoscedastic, as in Setting~1.

\begin{theorem}
\label{thm:bound}
Homoscedastic noise produces the best asymptotic PCA amplitude~\eqref{eq:rstheta}, subspace recovery~\eqref{eq:rs} and coefficient recovery~\eqref{eq:rsright} in Theorem~\ref{thm:rs} for a given average noise variance $\sigmab^2=p_1\sigma_1^2+\cdots+p_L\sigma_L^2$ over all distributions of noise variances for which $A(\beta_i) > 0$.
Namely, homoscedastic noise minimizes~\eqref{eq:rstheta} (and hence the positive bias) and it maximizes~\eqref{eq:rs} and~\eqref{eq:rsright}.
\end{theorem}

Concretely, suppose we had $c = 10$ samples per dimension and a subspace amplitude of $\tht_i = 1$.
Then the asymptotic subspace recoveries~\eqref{eq:rs} given in Theorem~\ref{thm:rs} evaluate to $0.818$ in Setting~1, $0.817$ in Setting~2 and $0$ in Setting~3; asymptotic recovery is best in Setting~1 as predicted by Theorem~\ref{thm:bound}.
Recovery is entirely lost in Setting~3, consistent with the observation that PCA is not robust to gross errors.
In Setting~2, only using the $1\%$ of samples with noise variance $0.01$ (resulting in $0.1$ samples per dimension) yields an asymptotic subspace recovery of $0.908$ and so we may hope that recovery with all data could be better.
Theorem~\ref{thm:bound} rigorously shows that PCA does not fully exploit these high quality samples and instead performs worse in Setting~2 than in Setting~1, if only slightly.

Section~\ref{sct:bound} presents the proof of Theorem~\ref{thm:bound}.
 It is notable that Theorem~\ref{thm:bound} holds for all  proportions~$p$, sample-to-dimension ratios~$c$ and subspace amplitudes $\tht_i$; there are no settings where PCA benefits from heteroscedastic noise over homoscedastic noise with the same average variance.
The following corollary is equivalent and provides an alternate way of viewing the result.

\begin{corollary}[Bounds on asymptotic recovery] \label{cor:bound}
If $A(\beta_i) \geq 0$ then the asymptotic PCA amplitude~\eqref{eq:rstheta} is bounded as
\begin{equation} \label{eq:thetabound}
\thh_i^2 \asto
\tht_i^2+\tht_i^2\left(\frac{\beta_i}{c\tht_i^2}-1\right)\left(\frac{\beta_i}{\tht_i^2}+1\right)
\geq \tht_i^2\left(1+\frac{\sigmab^2}{c\tht_i^2}\right)\left(1+\frac{\sigmab^2}{\tht_i^2}\right),
\end{equation}
the asymptotic subspace recovery~\eqref{eq:rs} is bounded as
\begin{equation} \label{eq:rsbound}
\left|\langle \uh_i,\Span\{\ut_j:\tht_j  =  \tht_i\} \rangle\right|^2
\asto \frac{A\left( \beta_{i}\right) }{\beta_{i}B_{i}^{\prime }\left( \beta
_{i}\right) }
\leq \frac{c-\sigmab^4/\tht_i^4}{c+\sigmab^2/\tht_i^2},
\end{equation}
and the asymptotic coefficient recovery~\eqref{eq:rsright} is bounded as
\begin{equation} \label{eq:rszbound}
\left|\left\langle \frac{\zh^{(i)}}{\sqrt{n}},\Span\{\zt^{(j)}:\tht_j  =  \tht_i\} \right\rangle\right|^2
\asto \frac{A(\beta_i)}{c \{ \beta_i+(1-c)\tht_i^2\} B_i'(\beta_i)} \leq \frac{c-\sigmab^4/\tht_i^4}{c(1+\sigmab^2/\tht_i^2)},
\end{equation}
where $\sigmab^2=p_1\sigma_1^2+\cdots+p_L\sigma_L^2$ is the average noise variance and the bounds are met with equality if and only if $\sigma_1^2 = \cdots = \sigma_L^2$.
\end{corollary}

\begin{proof}[Proof of Corollary~\ref{cor:bound}]
The bounds~\eqref{eq:thetabound},~\eqref{eq:rsbound} and~\eqref{eq:rszbound} follow immediately from Theorem~\ref{thm:bound} and the expressions for homoscedastic noise~\eqref{eq:rsthetah} and~\eqref{eq:rsh} in Section~\ref{sct:special}.
\end{proof}

\noindent Corollary~\ref{cor:bound} highlights that while average noise variance may be a practically convenient measure for the overall quality of data, it can lead to an overly optimistic estimate of the performance of PCA for heteroscedastic data. The expressions~\eqref{eq:rstheta},~\eqref{eq:rs} and~\eqref{eq:rsright} in Theorem~\ref{thm:rs} are more accurate.

\newcommand{\sigmat}{\mathcal{I}}
\begin{remark}[Average inverse noise variance] \label{rem:bound}
\em
Average inverse noise variance $\sigmat=p_1 \times1/\sigma_1^2+\cdots+p_L \times 1/\sigma_L^2$ is another natural measure for the overall quality of data.
In particular, it is the (scaled) Fisher information for heteroscedastic Gaussian measurements of a fixed scalar.
Theorem~\ref{thm:bound} implies that homoscedastic noise also produces the best asymptotic PCA performance for a given average inverse noise variance; note that homoscedastic noise minimizes the average noise variance in this case.
Thus, average inverse noise variance can also lead to an overly optimistic estimate of the performance of PCA for heteroscedastic data.
\end{remark}

\section{Impact of parameters} \label{sct:analysis}

The simplified expressions in Theorem~\ref{thm:rs} for the asymptotic performance of PCA provide insight into the impact of the model parameters: sample-to-dimension ratio~$c$, subspace amplitudes~$\tht_1,\dots,\tht_k$, proportions~$p_1,\dots,p_L$ and noise variances~$\sigma_1^2,\dots,\sigma_L^2$.
For brevity, we focus on the asymptotic subspace recovery~\eqref{eq:rs} of the $i$th component; similar phenomena occur for the asymptotic PCA amplitudes~\eqref{eq:rstheta} and coefficient recovery~\eqref{eq:rsright} as we show in Section~\ref{sct:analysis_right} of the Online Supplement.

\subsection{Impact of sample-to-dimension ratio~\texorpdfstring{$c$}{c} and subspace amplitude~\texorpdfstring{$\tht_i$}{theta\_i}} \label{sct:analysis_ctheta}

\begin{figure}[t]
\centering
\begin{subfigure}[t]{0.49\linewidth}
\centering
\includegraphics[width=\linewidth]{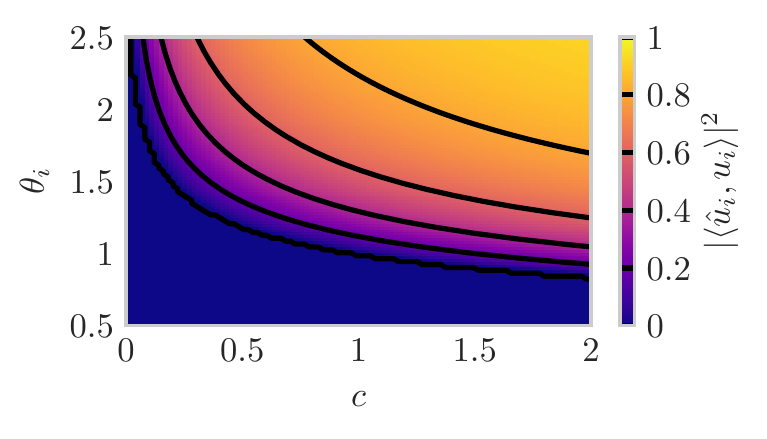}
\caption{Homoscedastic noise with $\sigma_1^2=1$.}
\label{fig:qual2a}
\end{subfigure}
\
\begin{subfigure}[t]{0.49\linewidth}
\centering
\includegraphics[width=\linewidth]{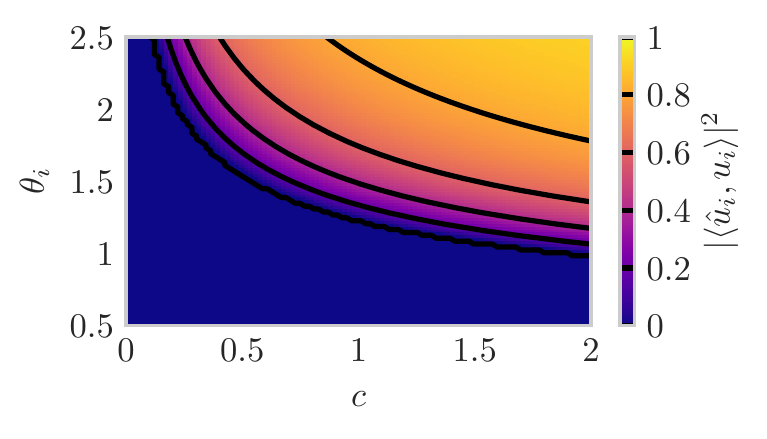}
\caption{Heteroscedastic noise with $p_1 = 80\%$ of samples at $\sigma_1^2=0.8$  and $p_2 = 20\%$ of samples at $\sigma_2^2=1.8$.}
\label{fig:qual2b}
\end{subfigure}
\caption{Asymptotic subspace recovery~\eqref{eq:rs} of the $i$th component as a function of sample-to-dimension ratio~$c$ and subspace amplitude~$\tht_i$ with average noise variance equal to one.
Contours are overlaid in black and the region where $A(\beta_i) \leq0$ is shown as zero (the prediction of Conjecture~\ref{conj:rs}). The phase transition in~(b) is further right than in (a);
more samples are needed to recover the same strength signal.}
\label{fig:qual2}
\end{figure}

Suppose first that there is only one noise variance fixed at $\sigma_1^2=1$ while we vary the sample-to-dimension ratio $c$ and subspace amplitude $\tht_i$.
This is the homoscedastic setting described in Section~\ref{sct:special}.
Figure~\ref{fig:qual2a} illustrates the expected behavior: decreasing the subspace amplitude~$\tht_i$ degrades asymptotic subspace recovery~\eqref{eq:rs} but the lost performance could be regained by increasing the number of samples. Figure~\ref{fig:qual2a} also illustrates a phase transition: a sufficient number of samples with a sufficiently large subspace amplitude is necessary to have an asymptotic recovery greater than zero.
Note that in all such figures, we label the axis $\left|\langle \uh_i,\ut_i\rangle\right|^2$ to indicate the asymptotic recovery on the right hand side of~\eqref{eq:rs}.

Now suppose that there are two noise variances~$\sigma_1^2=0.8$ and~$\sigma_2^2=1.8$ occurring in proportions~$p_1 = 80\%$ and $p_2 = 20\%$.
The average noise variance is still $1$, and Figure~\ref{fig:qual2b} illustrates similar overall features to the homoscedastic case. Decreasing subspace amplitude~$\tht_i$ once again degrades asymptotic subspace recovery~\eqref{eq:rs} and the lost performance could be regained by increasing the number of samples. However, the phase transition is further up and to the right compared to the homoscedastic case.
This is consistent with Theorem~\ref{thm:bound}; PCA performs worse on heteroscedastic data than it does on homoscedastic data of the same average noise variance, and thus more samples or a larger subspace amplitude are needed to recover the subspace basis element.

\subsection{Impact of proportions~\texorpdfstring{$p_1,\dots,p_L$}{p1,...,pL}} \label{sct:qual3}
\begin{figure}[t]
\centering
\includegraphics[width=0.45\linewidth]{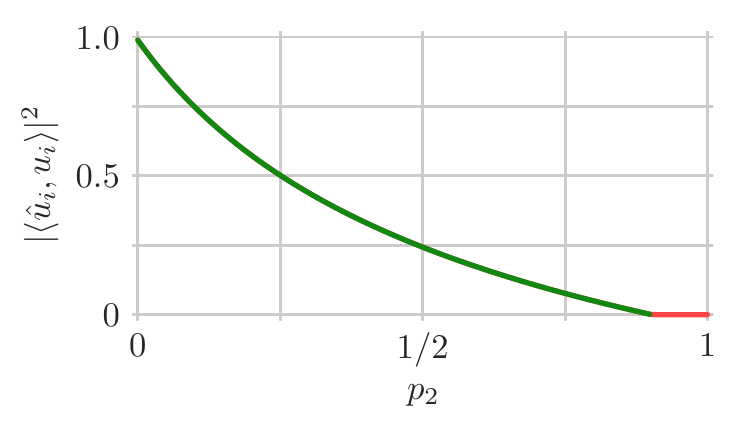}
\caption{
Asymptotic subspace recovery~\eqref{eq:rs} of the $i$th component as a function of the contamination fraction $p_2$, the proportion of samples with noise variance $\sigma_2^2 = 3.25$, where the other noise variance~$\sigma_1^2=0.1$ occurs in proportion $p_1 = 1-p_2$.
The sample-to-dimension ratio is $c=10$ and the subspace amplitude is $\tht_i=1$. The region where $A(\beta_i) \leq0$ is the red horizontal segment with  value zero (the prediction of Conjecture~\ref{conj:rs}).}
\label{fig:qual3}
\end{figure}

Suppose that there are two noise variances~$\sigma_1^2=0.1$ and~$\sigma_2^2=3.25$ occurring in proportions~$p_1=1-p_2$ and $p_2$, where
the sample-to-dimension ratio is $c=10$ and the subspace amplitude is~$\tht_i=1$.
Figure~\ref{fig:qual3} shows the asymptotic subspace recovery~\eqref{eq:rs} as a function of the proportion~$p_2$.
Since $\sigma_2^2$ is significantly larger, it is natural to think of $p_2$ as a fraction of contaminated samples.
As expected, performance generally degrades as $p_2$ increases and low noise samples with noise variance $\sigma_1^2$ are traded for high noise samples with noise variance $\sigma_2^2$.
The performance is best when $p_2 = 0$ and all the samples have the smaller noise variance~$\sigma_1^2$ (i.e., there is no contamination).

It is interesting that the asymptotic subspace recovery in Figure~\ref{fig:qual3} has a steeper slope initially for $p_2$ close to zero and then a shallower slope for $p_2$ close to one.
Thus the benefit of reducing the contamination fraction varies across the range.

\subsection{Impact of noise variances~\texorpdfstring{$\sigma_1^2,\dots,\sigma_L^2$}{sigma1\^{}2,...,sigmaL\^{}2}} \label{sct:impact_sigma}

\begin{figure}[t]
\centering
\includegraphics[width=0.45\linewidth]{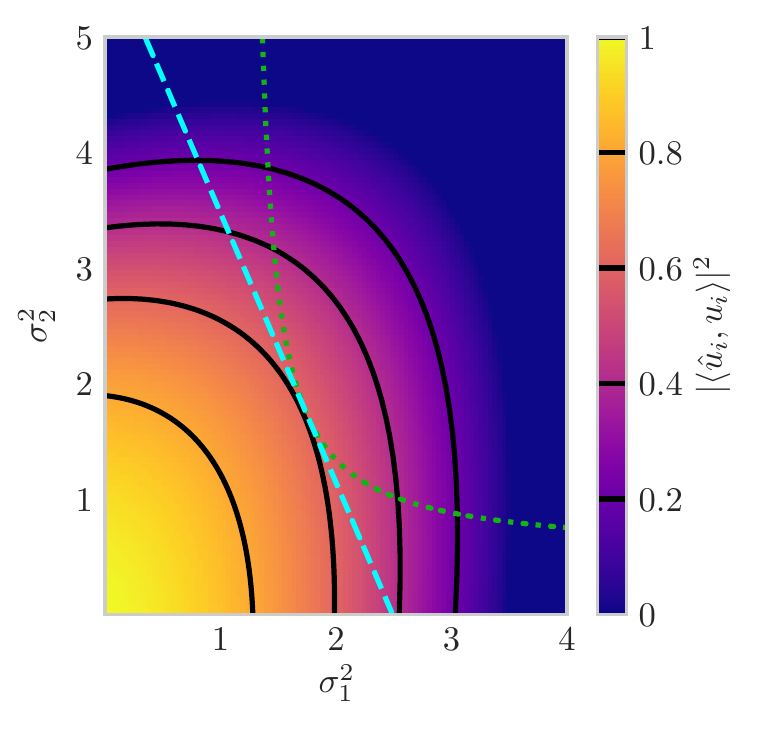}
\caption{Asymptotic subspace recovery~\eqref{eq:rs}  of the $i$th component as a function of noise variances~$\sigma_1^2$ and~$\sigma_2^2$ occurring in proportions~$p_1 = 70\%$ and $p_2 = 30\%$, where
the sample-to-dimension ratio is $c=10$ and the subspace amplitude is $\tht_i=1$. Contours are overlaid in black and the region where $A(\beta_i) \leq0$ is shown as zero (the prediction of Conjecture~\ref{conj:rs}).
Along the dashed cyan line, the average noise variance is $\sigmab^2\approx1.74$ and the best performance occurs when $\sigma_1^2=\sigma_2^2=\sigmab^2$.
Along the dotted green curve, the average inverse noise variance is $\sigmat\approx0.575$ and the best performance again occurs when $\sigma_1^2=\sigma_2^2$.}
\label{fig:qual1}
\end{figure}

Suppose that there are two noise variances~$\sigma_1^2$ and~$\sigma_2^2$ occurring in proportions~$p_1 = 70\%$ and $p_2 = 30\%$, where
the sample-to-dimension ratio is $c=10$ and the subspace amplitude is $\tht_i=1$.
Figure~\ref{fig:qual1} shows the asymptotic subspace recovery~\eqref{eq:rs} as a function of the noise variances~$\sigma_1^2$ and~$\sigma_2^2$.
As expected, performance typically degrades with increasing noise variances. However, there is a curious regime around $\sigma_1^2=0$ and $\sigma_2^2=4$ where increasing $\sigma_1^2$ slightly from zero improves asymptotic performance; the contour lines point slightly up and to the right.
We have also observed this phenomenon in finite-dimensional simulations, so this effect is not simply an asymptotic artifact.
This surprising phenomenon is an interesting avenue for future exploration.

The contours in Figure~\ref{fig:qual1} are generally horizontal for small $\sigma_1^2$ and vertical for small $\sigma_2^2$. This indicates that when the gap between the two largest noise variances is ``sufficiently'' wide, the asymptotic subspace recovery~\eqref{eq:rs} is roughly determined by the largest noise variance.
While initially unexpected, this property can be intuitively understood by recalling that~$\beta_i$ is the largest value of $x$ satisfying
\begin{equation}\label{eq:qualB}
\frac{1}{c\tht_i ^{2}}=\sum_{\ell =1}^{L}\frac{p_{\ell }}{x-\sigma_{\ell
}^{2}}
.
\end{equation}
When the gap between the two largest noise variances is wide, the largest noise variance is significantly larger than the rest and it dominates the sum in~\eqref{eq:qualB} for $x>\max_\ell (\sigma_\ell^2)$, i.e., where $\beta_i$ occurs.
Thus $\beta_i$, and similarly, $A(\beta_i)$ and $B_i'(\beta_i)$ are roughly determined by the largest noise variance.

The precise relative impact of each noise variance~$\sigma_\ell^2$ depends on its corresponding proportion~$p_\ell$, as shown by the asymmetry of Figure~\ref{fig:qual1} around the line $\sigma_1^2=\sigma_2^2$.
Nevertheless, very large noise variances can drown out the impact of small noise variances, regardless of their relative proportions.
This behavior provides a rough explanation for the sensitivity of PCA to even a few gross errors (i.e., outliers); even in  small proportions, sufficiently large errors dominate the performance of PCA.

Along the dashed cyan line in Figure~\ref{fig:qual1}, the average noise variance is $\sigmab^2\approx1.74$ and the best performance occurs when $\sigma_1^2=\sigma_2^2=\sigmab^2$,
as predicted by Theorem~\ref{thm:bound}.
Along the dotted green curve, the average inverse noise variance is $\sigmat\approx0.575$ and the best performance again occurs when $\sigma_1^2=\sigma_2^2$, as predicted in Remark~\ref{rem:bound}.
Note, in particular, that the dashed line and dotted curve are both tangent to the contour at exactly $\sigma_1^2=\sigma_2^2$.
The observation that larger noise variances have ``more impact'' provides a rough explanation for this phenomenon; homoscedasticity minimizes the largest noise variance for both the line and the curve.
In some sense, as discussed in Section~\ref{sct:optim}, the degradation from samples with larger noise is greater than the benefit of having samples
with correspondingly smaller noise.

\subsection{Impact of adding data} \label{sct:impact_add}
\begin{figure}[b]
\centering
\includegraphics[width=0.75\linewidth]{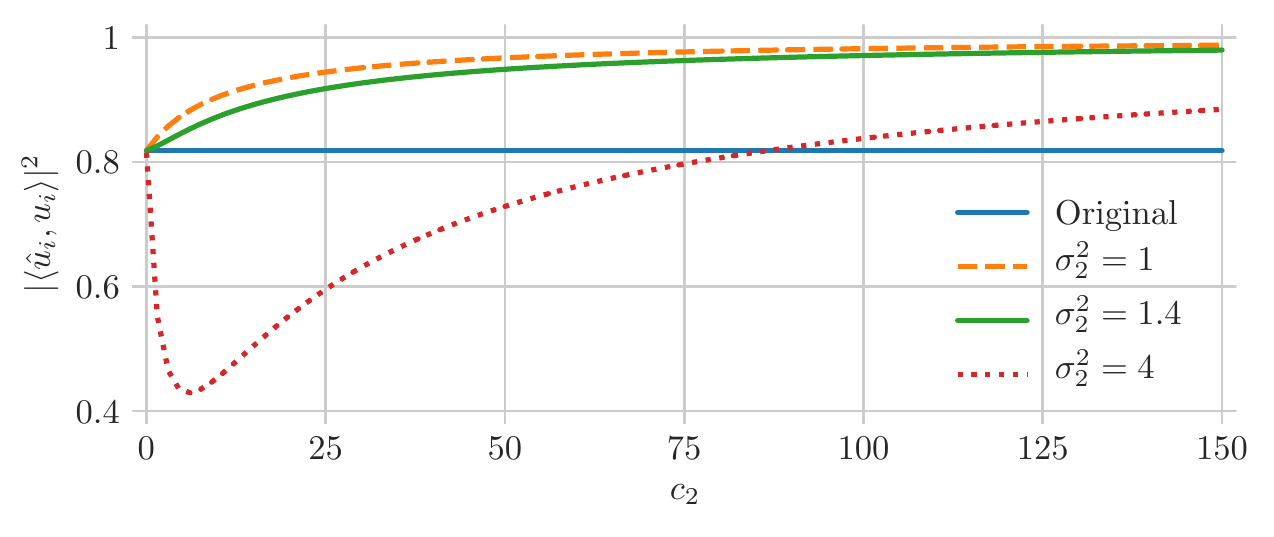}
\caption{Asymptotic subspace recovery~\eqref{eq:rs}  of the $i$th component for samples added with noise variance $\sigma_2^2$ and samples-per-dimension $c_2$ to an existing dataset with noise variance~$\sigma_1^2 = 1$, sample-to-dimension ratio $c_1=10$ and subspace amplitude $\tht_i=1$. }
\label{fig:impact_add}
\end{figure}

Consider adding data with noise variance~$\sigma_2^2$ and sample-to-dimension ratio~$c_2$ to an existing dataset that has noise variance~$\sigma_1^2 = 1$, sample-to-dimension ratio $c_1=10$ and subspace amplitude $\tht_i=1$  for the $i$th component.
The combined dataset has a sample-to-dimension ratio of $c=c_1+c_2$ and is potentially heteroscedastic with noise variances $\sigma_1^2$ and $\sigma_2^2$ appearing in proportions $p_1 = c_1/c$ and $p_2=c_2/c$.
Figure~\ref{fig:impact_add} shows the asymptotic subspace recovery~\eqref{eq:rs}  of the $i$th component for this combined dataset as a function of the sample-to-dimension ratio $c_2$ of the added data for a variety of noise variances $\sigma_2^2$.
The dashed orange curve, showing the recovery when $\sigma_2^2 = 1 = \sigma_1^2$, illustrates the benefit we would expect for homoscedastic data: increasing the samples per dimension improves recovery.
The dotted red curve shows the recovery when $\sigma_2^2=4 > \sigma_1^2$. For a small number of added samples, the harm of introducing noisier data outweighs the benefit of having more samples.
For sufficiently many samples, however, the tradeoff reverses and recovery for the combined dataset exceeds that for the original dataset; the break even point can be calculated using expression~\eqref{eq:rs}.
Finally, the green curve shows the performance when $\sigma_2^2=1.4 > \sigma_1^2$.
As before, the added samples are noisier than the original samples and so we might expect performance to initially decline again. In this case, however, the performance improves for any number of added samples.
In all three cases, the added samples dominate in the limit $c_2 \to \infty$ and PCA approaches perfect subspace recovery as one may expect.
However, perfect recovery in the limit does not typically happen for PCA amplitudes~\eqref{eq:rstheta} and coefficient recovery~\eqref{eq:rsright}; see Section~\ref{sct:add_tz} of the Online Supplement for more details.

Note that it is equivalent to think about removing noisy samples from a dataset by thinking of the combined dataset as the original full dataset. The green curve in Figure~\ref{fig:impact_add} then suggests that slightly noisier samples should not be removed; it would be best if the full data was homoscedastic but removing slightly noisier data (and reducing the dataset size) does more harm than good.
The dotted red curve in Figure~\ref{fig:impact_add} suggests that much noisier samples should be removed unless they are numerous enough to outweigh the cost of adding them.
Once again, expression~\eqref{eq:rs} can be used to calculate the break even point.

\section{Numerical simulation}

\label{sct:experiment}

\begin{figure}[b!]
\centering
\begin{subfigure}[t]{0.49\linewidth}
\centering
\includegraphics[width=0.95\linewidth]{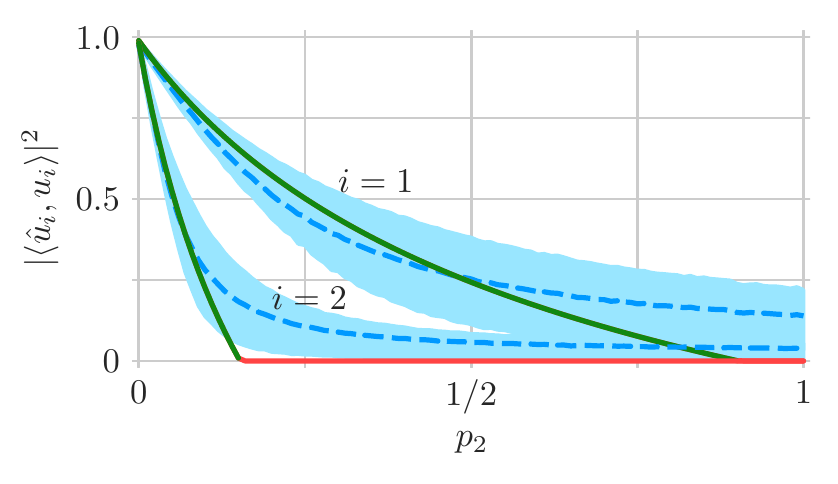}
\caption{$10^3$ samples in $10^2$ dimensions.}
\label{fig:exps51}
\end{subfigure}\
\begin{subfigure}[t]{0.49\linewidth}
\centering
\includegraphics[width=0.95\linewidth]{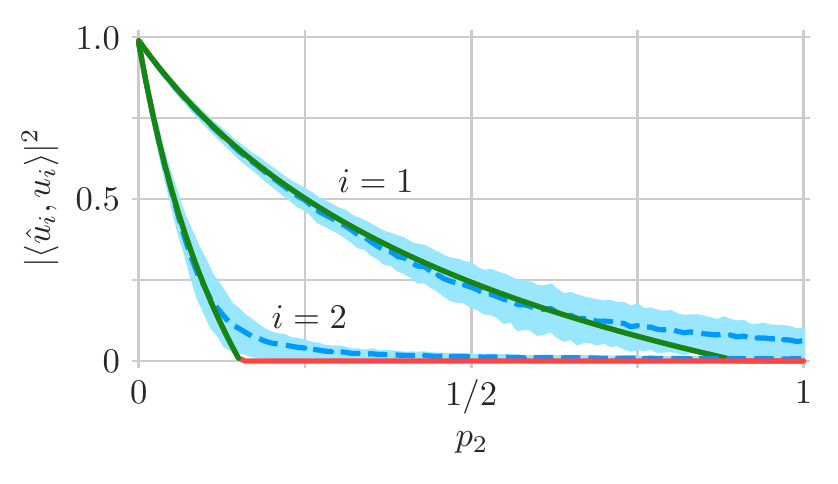}
\caption{$10^4$ samples in $10^3$ dimensions.}
\label{fig:exps52}
\end{subfigure}
\caption{Simulated subspace recovery~\eqref{eq:rs} as a function of the contamination fraction $p_2$, the proportion of samples with noise variance $\sigma_2^2 = 3.25$, where the other noise variance~$\sigma_1^2=0.1$ occurs in proportion $p_1 = 1-p_2$.
The sample-to-dimension ratio is $c=10$ and the subspace amplitudes are $\tht_1=1$ and $\tht_2=0.8$.
Simulation mean (dashed blue curve) and interquartile interval
(light blue ribbon) are shown with the asymptotic recovery~\eqref{eq:rs} of Theorem~\ref{thm:rs} (green curve).
The region where $A(\beta_i) \leq0$ is the red horizontal segment with  value zero (the prediction of Conjecture~\ref{conj:rs}).
Increasing data size from (a) to (b) results in  smaller interquartile intervals, indicating concentration to the mean, which is itself converging to the asymptotic recovery.
}
\label{fig:exps5}
\end{figure}

This section simulates data generated by the model described in Section~\ref{sct:model} to illustrate the main result, Theorem~\ref{thm:rs}, and to demonstrate that the asymptotic results provided are meaningful for practical settings with finitely many samples in a finite-dimensional space.
As in Section~\ref{sct:analysis}, we show results only for the asymptotic subspace recovery~\eqref{eq:rs} for brevity; the same phenomena occur for the asymptotic PCA amplitudes~\eqref{eq:rstheta} and coefficient recovery~\eqref{eq:rsright} as we show in Section~\ref{sct:exp_rest} of the Online Supplement.
Consider data from a two-dimensional subspace with subspace amplitudes $\tht_1 = 1$ and $\tht_2 = 0.8$, two noise variances $\sigma_1^2=0.1$ and~$\sigma_2^2=3.25$, and a sample-to-dimension ratio of $c=10$. We sweep the proportion of high noise samples $p_2$ from zero to one, setting $p_1=1-p_2$ as in Section~\ref{sct:qual3}.
The first simulation considers $n=10^{3}$ samples in a $d=10^{2}$ dimensional ambient space ($10^4$ trials). The second increases these to $n=10^{4}$ samples in a $d=10^{3}$ dimensional ambient space ($10^3$ trials).
Both simulations generate data from the standard normal distribution, i.e., $\zt_{ij},\varepsilon_{ij}\sim\cN(0,1)$.
Note that sweeping over $p_2$ covers homoscedastic settings at the extremes ($p_{2}=0,1$) and evenly split heteroscedastic data in the middle ($p_{2}=1/2$).

Figure~\ref{fig:exps5} plots the recovery of subspace components $|\langle\uh_i,\ut_i\rangle|^2$ for both simulations with the mean (dashed blue curve) and interquartile interval (light blue ribbon) shown with the asymptotic recovery~\eqref{eq:rs} of Theorem~\ref{thm:rs} (green curve).
The region where $A(\beta_i) \leq0$ is the red horizontal segment with  value zero (the prediction of Conjecture~\ref{conj:rs}).
Figure~\ref{fig:exps51} illustrates general agreement between the mean and the asymptotic recovery, especially far away from the non-differentiable points where the recovery becomes zero and Conjecture~\ref{conj:rs} predicts a phase transition.
This is a general phenomenon we observed: near the phase transition the smooth simulation mean deviates from the non-smooth asymptotic recovery.
Intuitively, an asymptotic recovery of zero corresponds to PCA components that are like isotropically random vectors and so have vanishing
square inner product with the true components as the dimension grows.
In finite dimension, however, there is a chance of alignment that results in a positive square inner product.

Figure~\ref{fig:exps52} shows what happens when the number of samples and ambient dimension are increased to $n=10^{4}$ and $d=10^{3}$.
The interquartile intervals are roughly half the size of those in Figure~\ref{fig:exps51}, indicating concentration of the recovery of each component (a random quantity) around its mean. Furthermore, there is better agreement between the mean and the asymptotic recovery, with the maximum deviation between simulation and asymptotic prediction still occurring nearby the phase transition.
In particular for $p_2 < 0.75$ the largest deviation for $|\langle \uh_1,\ut_1\rangle|^2$ is around $0.03$. For $p_2 \notin (0.1,0.35)$, the largest deviation for $|\langle \uh_2,\ut_2\rangle|^2$ is around $0.02$.
To summarize, the numerical simulations indicate that the subspace recovery concentrates to its mean and that the mean approaches the asymptotic recovery.
Furthermore, good agreement with Conjecture~\ref{conj:rs} provides further evidence that there is indeed a phase transition below which the subspace is not recovered.
These findings are similar to those in~\cite{hong2016tat} for a one-dimensional subspace with two noise variances.

\section{Proof of Theorem~\ref{thm:rs}} \label{sct:proof}
The proof has six main parts. Section~\ref{sct:initial} connects several results from random matrix theory to obtain an initial expression for asymptotic recovery. This expression is difficult to evaluate and
analyze because it involves an integral transform of the (nontrivial)
limiting singular value distribution for a random (noise) matrix as well as
the corresponding limiting largest singular value.
However, we have discovered a nontrivial structure in this expression that enables us to derive a much simpler form in Sections \ref{sct:varchange}-\ref{sct:r_alg}.

\subsection{Obtain an initial expression}

\label{sct:initial}

Rewriting the model in~\eqref{eq:model} in matrix form yields
\begin{equation} \label{eq:matmodel}
\Y = (y_1, \ldots, y_n) =\Ut\Tht\Zt^\CT+\E\Eta\in \bC^{d\times n},
\end{equation}
where
\begin{itemize}
\item[]
$\Zt = (\zt^{(1)},\ldots,\zt^{(k)}) \in \bC^{n\times k}$ is the coefficient matrix,
\item[]
$\E = (\varepsilon_1, \ldots,\varepsilon_n) \in \bC^{d\times n}$ is the (unscaled) noise matrix,
\item[]
$\Eta = \diag(\eta_{1},\dots,\eta_{n}) \in \bR_+^{n\times n}$ is a diagonal matrix of noise standard deviations.
\end{itemize}
The first $k$ principal components~$\uh_1, \ldots, \uh_k$, PCA amplitudes~$\thh_1, \ldots, \thh_k$ and (normalized) scores~$\zh^{(1)}/\sqrt{n}, \ldots, \zh^{(k)}/\sqrt{n}$ defined in Section~\ref{sct:intro} are exactly the first $k$ left singular vectors, singular values and right singular vectors, respectively, of the scaled data matrix $\Y / \sqrt{n}$.

To match the model of~\cite{benaych2012tsv}, we introduce the random unitary matrix
\begin{equation*}
\R = [
\begin{array}{cc}
\Ub & \Ub^{\perp }
\end{array}
] [
\begin{array}{cc}
\Ut & \Ut^{\perp }
\end{array}
] ^{\CT }
=
\Ub\Ut^\CT
+ \Ub^{\perp }(\Ut^{\perp })^\CT,
\end{equation*}
where the random matrix $\Ub\in \bC^{d\times k}$ is the Gram-Schmidt orthonormalization of a $d\times k$ random matrix that has iid (mean zero,
variance one) circularly symmetric complex normal $\cCN(0,1)$ entries.
We use the superscript $\perp$ to denote a matrix of orthonormal basis elements for the orthogonal complement; the columns of $\Ut^\perp$ form an orthonormal basis for the orthogonal complement of the column span of $\Ut$.

Left multiplying~\eqref{eq:matmodel} by $\R /\sqrt{n}$ yields that  $\R\uh_1, \ldots, \R\uh_k$, $\thh_1, \ldots, \thh_k$ and $\zh^{(1)}/\sqrt{n}, \ldots, \zh^{(k)}/\sqrt{n}$ are the first $k$ left singular vectors, singular values and right singular vectors, respectively, of the scaled and rotated data matrix
\begin{equation*}
\Yn = \frac{1}{\sqrt{n}}\, \R\Y.
\end{equation*}
The matrix $\Yn$ matches the low rank (i.e., rank $k$)
perturbation of a random matrix model considered in\ \cite{benaych2012tsv}
because
\begin{equation*}
\Yn=\LP+\X,
\end{equation*}
where
\begin{align*}
\LP = \frac{1}{\sqrt{n}}\, \R\left( \Ut\Tht\Zt^\CT\right)
=\frac{1}{\sqrt{n}}\,
\Ub\Tht\Zt^\CT=\sum_{i=1}^{k}\tht_{i}\ub_{i}\left( \frac{1}{\sqrt{n}}
\, \zt^{\left( i\right) }\right)
^{\CT } , \quad
\X = \frac{1}{\sqrt{n}}\, \R\left( \E\Eta\right) =\left(
\frac{1}{\sqrt{n}}\, \R\E\right) \Eta.
\end{align*}
Here $\LP$ is generated according to the \textquotedblleft
orthonormalized model\textquotedblright\ in~\cite{benaych2012tsv} for the vectors $\ub_{i}$ and
the \textquotedblleft iid model\textquotedblright\ for the vectors $\zt^{\left( i\right) }$ and $\LP$ satisfies Assumption 2.4 of~\cite{benaych2012tsv}; the latter considers $\ub_{i}$ and $\zt^{\left( i\right) }$
to be generated according to the same model, but its proof extends to this case. Furthermore $\R\E$ has iid
entries with zero mean, unit variance and bounded fourth moment
(by the assumption that $\varepsilon_i$ are unitarily invariant), and $\Eta$ is a non-random diagonal positive definite matrix with bounded spectral norm and limiting eigenvalue distribution $p_1\delta_{\sigma_1^2}+\cdots+p_L\delta_{\sigma_L^2}$, where $\delta_{\sigma_\ell^2}$ is the Dirac delta distribution centered at $\sigma_\ell^2$.
Under these conditions, Theorem~4.3 and Corollary~6.6 of~\cite{bai2010sao} state that $\X$ has a non-random compactly supported limiting singular value distribution $\mu_\X$ and the largest singular value of $\X$ converges almost surely to the supremum of the support of $\mu_\X$.
Thus Assumptions~2.1 and 2.3 of~\cite{benaych2012tsv} are also satisfied.

Furthermore,
$
\uh_i^{\CT}\ut_j
=\uh_i^{\CT}\R^\CT\R\ut_j
=(\R\uh_i)^\CT\ub_j
$
for all $i,j \in \{1,\dots,k\}$ so
\begin{align*}
|\langle \R\uh_i,\Span\{\ub_j:\tht_j   =  \tht_i\} \rangle|^2
&= |\langle \uh_i,\Span\{\ut_j:\tht_j   =  \tht_i\} \rangle|^2, \\
|\langle \R\uh_i,\Span\{\ub_j:\tht_j \neq \tht_i\} \rangle|^2
&= |\langle \uh_i,\Span\{\ut_j:\tht_j \neq \tht_i\} \rangle|^2,
\end{align*}
and hence Theorem 2.10 from \cite{benaych2012tsv} implies that, for each $i \in \{1,\ldots,k\}$,
\begin{equation} \label{eq:initialtheta}
\thh_i^2 \asto
\begin{cases}
\rho_i^2 & \text{if } \tht_i^2 > \thb^2,\\
b^2 & \text{otherwise,}
\end{cases}
\end{equation}
and that if $\tht_i^2 >\thb^2$, then
\begin{align}
|\langle \uh_i,\Span\{\ut_j:\tht_j   =  \tht_i\} \rangle|^2
&\asto \frac{-2 \varphi(\rho_i)}{\tht_i^2 D'(\rho_i) } \label{eq:init} , \\
\left|\left\langle \frac{\zh^{(i)}}{\sqrt{n}},\Span\{\zt^{(j)}:\tht_j   =  \tht_i\} \right\rangle\right|^2
&\asto \frac{-2 \{c^{-1}\varphi(\rho_i)+(1-c^{-1})/\rho_i\}}{\tht_i^2 D'(\rho_i) } , \nonumber
\end{align}
and
\begin{align}
|\langle \uh_i,\Span\{\ut_j:\tht_j \neq \tht_i\} \rangle|^2
&\asto 0, \label{eq:initzero} \\
\left|\left\langle \frac{\zh^{(i)}}{\sqrt{n}},\Span\{\zt^{(j)}:\tht_j \neq \tht_i\} \right\rangle\right|^2 &\asto 0 , \nonumber
\end{align}
where
$\rho_{i} = D^{-1}( 1/\tht_{i}^{2}) $, $\thb^{2} = 1/D\left( b^{+}\right)
$,
$D (z)  = \varphi (z) \{ c^{-1}\varphi
\left( z\right) +
(1-c^{-1})/z\}$
for $z>b$,
$
\varphi \left( z\right) = \int z/(z^{2}-t^{2})\ d\mu_{\X}\left(
t\right),
$
$b$ is the supremum of the support of $\mu_{\X}$ and $\mu_{\X}$ is the limiting singular value distribution of $\X$ (compactly supported by Assumption 2.1 of~\cite{benaych2012tsv}).
We use the notation $f\left( b^{+}\right) = \lim_{z\rightarrow
b^{+}}f\left( z\right) $ as a convenient shorthand for the limit from above of a
function $f\left( z\right) $.

Theorem~2.10 from~\cite{benaych2012tsv} is presented therein for $d \leq n$ (i.e., $c \geq 1$) to simplify their proofs. However, it also holds without modification for $d > n$ if the limiting singular value distribution~$\mu_\X$ is always taken to be the limit of the empirical distribution of the  $d$ largest singular values ($d-n$ of which will be zero). Thus we proceed without the condition that $c > 1$.

Furthermore, even though it is not explicitly stated as a main result in~\cite{benaych2012tsv}, the proof of Theorem~2.10 in~\cite{benaych2012tsv} implies that
\begin{equation} \label{eq:initmix}
\sum_{j:\tht_j=\tht_i}
\langle \uh_i,\ut_j \rangle
\left\langle \frac{\zh^{(i)}}{\sqrt{n}},\frac{\zt^{(j)}}{\|\zt^{(j)}\|} \right\rangle^*
\asto
\sqrt{
\frac{-2 \varphi(\rho_i)}{\tht_i^2 D'(\rho_i) }
\times
\frac{-2 \{c^{-1}\varphi(\rho_i)+(1-c^{-1})/\rho_i\}}{\tht_i^2 D'(\rho_i)},
}
\end{equation}
as was also noted in~\cite{nadakuditi2014oaa} for the special case of distinct subspace amplitudes.

Evaluating the expressions~\eqref{eq:initialtheta},~\eqref{eq:init} and~\eqref{eq:initmix} would consist of evaluating the intermediates listed above from last to first. These steps are challenging because they involve an integral transform of the limiting singular value
distribution $\mu_\X$ for the random (noise) matrix $\X$ as well as the corresponding
limiting largest singular value $b$, both of which depend nontrivially on the model parameters.
Our analysis uncovers a nontrivial structure that we exploit to derive simpler expressions.

Before proceeding, observe that the almost sure limit in~\eqref{eq:initmix} is just the geometric mean of the two almost sure limits in~\eqref{eq:init}.
Hence, we proceed to derive simplified expressions for~\eqref{eq:initialtheta} and~\eqref{eq:init}; \eqref{eq:rsmix} follows as the geometric mean of the simplified expressions obtained for the almost sure limits in~\eqref{eq:init}.

\subsection{Perform a change of variables}

\label{sct:varchange}

We introduce the function defined, for $z>b$, by
\begin{equation}  \label{eq:psi}
\psi \left( z\right) = \frac{cz}{\varphi \left( z\right) }=\left\{ \frac{1}{c}\int\frac{1}{z^{2}-t^{2}}d\mu_{\X}\left( t\right)
\right\}
^{-1},
\end{equation}
because it turns out to have several nice properties that simplify all of
the following analysis.
Rewriting~\eqref{eq:init} using $\psi \left( z\right) $ instead of $\varphi
\left( z\right) $ and factoring appropriately yields that if $\tht_i^2  >\thb^2$ then
\begin{align}
|\langle \uh_i,\Span\{\ut_j:\tht_j   =  \tht_i\} \rangle|^2
&\asto \frac{1}{\psi(\rho_i)}\frac{-2c}{\tht_i^2 D'(\rho_i)/\rho_i} \label{eq:rspsi} , \\
\left|\left\langle \frac{\zh^{(i)}}{\sqrt{n}},\Span\{\zt^{(j)}:\tht_j   =  \tht_i\} \right\rangle\right|^2
&\asto \frac{1}{c \{ \psi(\rho_i)+(1-c)\tht_i^2\}}\frac{-2c}{\tht_i^2 D'(\rho_i)/\rho_i } , \nonumber
\end{align}
where now
\begin{equation} \label{eq:Dpsi}
D\left( z\right) =\frac{cz^{2}}{\psi^2 \left( z\right)}+\frac{c-1}{\psi \left( z\right) }
\end{equation}
for $z > b$ and we have used the fact that
\begin{equation*}
\frac{1}{c}\left\{ \frac{1}{\psi(\rho_i)}+\frac{1-c^{-1}}{\rho_i^2}\right\}
= \frac{1}{c}\left\{ \psi(\rho_i) + \frac{1-c}{D(\rho_i)}\right\}^{-1}
= \frac{1}{c \{ \psi(\rho_i) + (1-c)\tht_i^2\}}.
\end{equation*}

\subsection{Find useful properties of~\texorpdfstring{$\psi(z)$}{psi(z)}}

\label{sct:props}

Establishing some properties of $\psi \left( z\right) $ aids
simplification significantly.

\medskip

\noindent \textbf{Property 1.}
We show that $\psi \left( z\right) $
satisfies a certain rational equation for all $z>b$ and derive its inverse function $\psi^{-1}(x)$.
Observe that the square singular values of the noise matrix $\X$ are exactly the eigenvalues of $c\X\X^{\CT }$, divided by $c$. Thus we first consider the
limiting eigenvalue distribution $\mu_{c\X\X^{\CT }}$ of $c\X\X^{\CT }$ and then relate its Stieltjes transform $m\left(
\zeta \right) $ to $\psi \left( z\right) $.

Theorem~4.3 in~\cite{bai2010sao} establishes that the random matrix
$
c\X\X^{\CT }=(1/d)\E
\Eta ^{2}\E^{\CT }
$
has a limiting eigenvalue distribution $\mu_{c\X\X^{\CT }}$
whose Stieltjes transform is given, for $\zeta \in \bC^{+}$, by
\begin{equation}  \label{eq:stieltjes}
m(\zeta) = \int \frac{1}{t-\zeta }d\mu_{c\X\X^{\CT }}\left( t\right) ,
\end{equation}
and satisfies the condition
\begin{equation}  \label{eq:pan}
\forall_{\zeta \in \bC^{+}}\quad m\left( \zeta \right)
=-\left\{ \zeta
-c\sum_{\ell =1}^{L}\frac{p_{\ell }\sigma_{\ell }^{2}}{1+\sigma_{\ell
}^{2}m\left( \zeta \right) }\right\} ^{-1},
\end{equation}
where $\bC^{+}$ is the set of all complex numbers with positive
imaginary part.

Since the $d$ square singular values of $\X$ are exactly the $d$ eigenvalues
of $c\X\X^{\CT }$ divided by $c$, we have for all $z > b$
\begin{equation}
\psi \left( z\right)
=\left\{ \frac{1}{c}\int\frac{1}{z^{2}-t^{2}}d\mu_{\X}\left( t\right) \right\} ^{-1}
=-\left\{ \int\frac{1}{t-z^{2}c}d\mu_{c\X\X^{\CT }}\left( t\right) \right\} ^{-1}
.  \label{eq:X_cXXt}
\end{equation}
For all $z$ and $\xi >0$, $z^{2}c+i\xi \in \bC^{+}$ and so combining~\eqref{eq:stieltjes}--\eqref{eq:X_cXXt} yields that for all $z>b$
\begin{equation*}
\psi \left( z\right)
=-\left\{ \lim_{\xi \rightarrow 0^{+}}m (z^{2}c+i\xi  ) \right\} ^{-1} 
=z^{2}c-c\sum_{\ell =1}^{L}\frac{p_{\ell }\sigma_{\ell }^{2}}{1-\sigma
_{\ell }^{2}/\psi \left( z\right) }.
\end{equation*}
Rearranging yields
\begin{equation}  \label{eq:Q1}
0=\frac{cz^{2}}{ \psi^2 \left( z\right) }-\frac{1}{\psi \left( z\right) }-\frac{c}{\psi \left( z\right) }\sum_{\ell
=1}^{L}\frac{p_{\ell }\sigma_{\ell }^{2}}{\psi \left( z\right) -\sigma
_{\ell }^{2}},
\end{equation}
for all $z > b$, where the last term is
\begin{equation*}
-\frac{c}{\psi \left( z\right) }\sum_{\ell =1}^{L}\frac{p_{\ell }\sigma
_{\ell }^{2}}{\psi \left( z\right) -\sigma_{\ell }^{2}}
=\frac{c}{\psi \left( z\right) }-c\sum_{\ell =1}^{L}\frac{p_{\ell }}{\psi
\left( z\right) -\sigma_{\ell }^{2}},
\end{equation*}
because $p_{1}+\cdots +p_{L}=1$. Substituting back into~\eqref{eq:Q1}
yields $0=Q\{ \psi \left( z\right) ,z\} $ for all $z>b$,
where
\begin{equation}  \label{eq:Q}
Q\left( s,z\right) = \frac{cz^{2}}{s^{2}}+\frac{c-1}{s}-c\sum_{\ell =1}^{L}\frac{p_{\ell }}{s-\sigma_{\ell }^{2}} .
\end{equation}
Thus $\psi(z) $ is an algebraic function (the associated polynomial can be formed by clearing the denominator of $Q$).
Solving~\eqref{eq:Q} for $z>b$ yields the inverse
\begin{equation} \label{eq:psiinv}
\psi^{-1}(x) = \sqrt{\frac{1-c}{c}x+x^2\sum_{\ell=1}^L \frac{p_\ell}{x-\sigma_\ell^2}}
= \sqrt{\frac{x}{c}\left(1+c\sum_{\ell=1}^L \frac{p_\ell\sigma_\ell^2}{x-\sigma_\ell^2}\right)}.
\end{equation}

\medskip

\noindent \textbf{Property 2.}
We show that $\max_\ell ( \sigma_\ell^2) <\psi(z)<cz^2$ for $z > b$.
For $z>b$, one can show from~\eqref{eq:psi} that $\psi(y)$ increases continuously and monotonically from $\psi(z)$ to infinity as $y$ increases from $z$ to infinity, and hence $\psi^{-1}(x)$ must increase continuously and monotonically from $z$ to infinity as $x$ increases from $\psi(z)$ to infinity.
However, $\psi^{-1}(x)$ is discontinuous at $x=\max_\ell ( \sigma_\ell^2) $ because $\psi^{-1}(x) \to \infty$ as $x \to \max_\ell (\sigma_\ell^2)$ from the right, and so it follows that $\psi(z) > \max_\ell (\sigma_\ell^2)$.
Thus $1/ \{ \psi(z)-\sigma_{\ell }^{2}\} > 0$ for all $\ell \in \{1,\dots,L\}$ and so
\begin{equation*}
cz^2=c [\psi^{-1}\{\psi(z)\}]^2
=\psi(z)\left\{1+c\sum_{\ell=1}^L \frac{p_\ell\sigma_\ell^2}{\psi(z)-\sigma_\ell^2}\right\}
>\psi(z).
\end{equation*}

\noindent \textbf{Property 3.}
We show that $0 < \psi \left( b^{+}\right) < \infty$ and $\psi ^{\prime }\left( b^{+}\right) =\infty $.
Property~2 in the limit $z=b^+$ implies that
\begin{equation*}
0 < \max_\ell (\sigma_\ell^2) \leq \psi(b^+) \leq cb^2 < \infty.
\end{equation*}
Taking the total derivative of $0=Q\{\psi(z),z\}$ with respect to $z$ and solving for $\psi'(z)$ yields
\begin{equation} \label{eq:psidiff}
\psi ^{\prime }\left( z\right) =- \frac{\partial Q}{\partial z}\{
\psi \left( z\right) ,z\} {\Big /} \frac{\partial Q}{\partial s}\{\psi
\left( z\right) ,z\} .
\end{equation}
As observed in~\cite{nadakuditi2006tpm}, the non-pole boundary points of compactly supported
distributions like $\mu_{c\X\X^{\CT }}$ occur where the
polynomial defining their Stieltjes transform has multiple roots.
Thus $\psi(b^+)$ is a multiple root of $Q(\cdot,b)$ and so
\begin{align*}
\frac{\partial Q}{\partial s}\{\psi \left( b^{+}\right) ,b\} =0 , \quad
\frac{\partial Q}{\partial z}\{\psi \left( b^{+}\right) ,b\} =\frac{2cb}{ \psi^2 \left( b^{+}\right) }>0.
\end{align*}
Thus $\psi ^{\prime }\left( b^{+}\right) =\infty $, where the sign is
positive because $\psi \left( z\right) $ is an increasing
function on $z>b$.

Summarizing, we have shown that
\begin{enumerate}
\item [a)]
$0=Q\{ \psi \left( z\right)
,z\}$ for all $z>b$ where $Q$ is defined in~\eqref{eq:Q}, and the inverse function $\psi^{-1}(x)$ is given in~\eqref{eq:psiinv},
\item [b)]
$\max_\ell (\sigma_\ell^2) <\psi(z)<cz^2$,
\item [c)]
$0 < \psi \left( b^{+}\right) < \infty$ and $\psi ^{\prime }\left(
b^{+}\right) =\infty $.
\end{enumerate}
\noindent We now use these properties to aid simplification.

\subsection{Express~\texorpdfstring{$D(z)$}{D(z)} and~\texorpdfstring{$D'(z)/z$}{D'(z)/z} in terms of only~\texorpdfstring{$\psi(z)$}{psi(z)}}

\label{sct:DDp_psi}

We can rewrite~\eqref{eq:Dpsi} as
\begin{equation}  \label{eq:D}
D\left( z\right)
=Q\{ \psi (z) ,z\} +c\sum_{\ell =1}^{L}\frac{p_{\ell }}{\psi \left( z\right) -\sigma_{\ell}^{2}}
=c\sum_{\ell =1}^{L}\frac{p_{\ell }}{\psi \left( z\right)
-\sigma_{\ell }^{2}} .
\end{equation}
because $0=Q\{ \psi \left( z\right) ,z\}$ by Property~1 of Section~\ref{sct:props}.
Differentiating~\eqref{eq:D} with respect to $z$
yields
\begin{equation*}
D^{\prime }\left( z\right) =-c\psi ^{\prime }\left( z\right) \sum_{\ell
=1}^{L}\frac{p_{\ell }}{\{ \psi (z) -\sigma_{\ell
}^{2}\} ^{2}},
\end{equation*}
and so we need to find $\psi ^{\prime }\left( z\right) $ in terms of $\psi \left( z\right) $.
Substituting the expressions for the partial derivatives $\partial Q \{
\psi \left( z\right) ,z\}/\partial z$ and $\partial Q\{\psi \left( z\right) ,z\}/\partial s$ into~\eqref{eq:psidiff} and simplifying we obtain
$\psi ^{\prime }\left( z\right) =2cz/\gamma \left( z\right) $,
where the denominator is
\begin{equation*}
\gamma \left( z\right) = c-1+\frac{2cz^{2}}{\psi \left( z\right) }-c\sum_{\ell =1}^{L}\frac{p_{\ell }\psi^2 \left( z\right) }{\{ \psi \left( z\right) -\sigma_{\ell }^{2}\} ^{2}} .
\end{equation*}
Note that
\begin{equation*}
\frac{2cz^{2}}{\psi \left( z\right) }=-2\left( c-1\right) +c\sum_{\ell
=1}^{L}\frac{2p_{\ell }\psi \left( z\right) }{\psi \left( z\right) -\sigma
_{\ell }^{2}},
\end{equation*}
because $0=Q\{ \psi \left( z\right) ,z\}$ for $z>b$. Substituting
into $\gamma(z)$ and forming a common denominator, then dividing with respect to $\psi(z)$ yields
\begin{equation*}
\gamma \left( z\right)
=1-c +c\sum_{\ell =1}^{L}p_{\ell }\frac{\psi^2 \left(z\right)-2\psi \left( z\right) \sigma_{\ell }^{2}}{\{ \psi \left( z\right)
-\sigma_{\ell }^{2}\} ^{2}}
=1-c\sum_{\ell =1}^{L}\frac{p_{\ell }\sigma_{\ell }^{4}}{\{ \psi\left(
z\right) -\sigma_{\ell }^{2}\} ^{2}} =A\{ \psi
\left(z\right)\},
\end{equation*}
where $A(x)$ was defined in~\eqref{eq:ABdef}.
Thus
\begin{equation}  \label{eq:psip}
\psi ^{\prime }\left( z\right) =\frac{2cz}{A\{ \psi \left( z\right)
\} },
\end{equation}
and
\begin{equation}  \label{eq:Dp}
\frac{D^{\prime }\left( z\right) }{z}=-\frac{2c^{2}}{A\{ \psi \left(
z\right) \} }\sum_{\ell =1}^{L}\frac{p_{\ell }}{\{ \psi \left(
z\right) -\sigma_{\ell }^{2}\} ^{2}}
=-\frac{2c}{\tht_i^2}\frac{B_i'\{\psi(z)\}}{A\{ \psi \left(
z\right) \} },
\end{equation}
where $B_i'(x)$ is the derivative of $B_i(x)$ defined in~\eqref{eq:ABdef}.

\subsection{Express the asymptotic recoveries in terms of only~\texorpdfstring{$\psi(b^+)$}{psi(b+)} and~\texorpdfstring{$\psi(\rho_i)$}{psi(rho\_i)}}

\label{sct:r_psi}
Evaluating~\eqref{eq:D} in the limit $z=b^+$ and recalling that $D(b^+)=1/\thb^2$ yields
\begin{equation} \label{eq:cond_brho}
\tht_{i}^{2}>\thb^{2} \quad \Leftrightarrow \quad
0>1-\frac{\tht_{i}^{2}}{\thb^{2}}=1-c\tht_{i}^{2}\sum_{\ell
=1}^{L}\frac{p_{\ell }}{\psi \left( b^{+}\right) -\sigma_{\ell }^{2}}=B_{i}\left\{ \psi \left( b^{+}\right) \right\},
\end{equation}
where $B_i(x)$ was defined in~\eqref{eq:ABdef}. Evaluating the inverse function~\eqref{eq:psiinv} both for $\psi(\rho_i)$ and in the limit $\psi(b^+)$ then substituting into~\eqref{eq:initialtheta} yields
\begin{equation} \label{eq:theta_brho}
\thh_i^2
\asto
\begin{cases}
\displaystyle \frac{\psi(\rho_i)}{c}\left\{1+c\sum_{\ell=1}^L\frac{p_\ell\sigma_\ell^2}{\psi(\rho_i)-\sigma_\ell^2}\right\}
& \text{if } B_i \{ \psi(b^+)\} < 0,\\
\displaystyle \frac{\psi(b^+)}{c}\left\{1+c\sum_{\ell=1}^L\frac{p_\ell\sigma_\ell^2}{\psi(b^+)-\sigma_\ell^2}\right\}
 & \text{otherwise.}
\end{cases}
\end{equation}
Evaluating~\eqref{eq:Dp} for $z=\rho_i$ and substituting into~\eqref{eq:rspsi} yields
\begin{align} \label{eq:rs_brho}
|\langle \uh_i,\Span\{\ut_j:\tht_j   =  \tht_i\} \rangle|^2
&\asto \frac{1}{\psi(\rho_i)}\frac{A\{ \psi(\rho_i)\}}{B_i'\{ \psi(\rho_i)\} } ,\\
\left|\left\langle \frac{\zh^{(i)}}{\sqrt{n}},\Span\{\zt^{(j)}:\tht_j   =  \tht_i\} \right\rangle\right|^2
&\asto \frac{1}{c\{\psi(\rho_i)+(1-c)\tht_i^2\}}\frac{A\{ \psi(\rho_i)\}}{B_i' \{ \psi(\rho_i)\}}, \nonumber
\end{align}
if $B_{i} \{ \psi \left( b^{+}\right) \} <0$.

\subsection{Obtain algebraic descriptions}

\label{sct:r_alg}

This subsection obtains algebraic descriptions of~\eqref{eq:cond_brho},~\eqref{eq:theta_brho} and~\eqref{eq:rs_brho} by showing that $\psi(b^+)$ is the largest real root of $A(x)$ and that $\psi(\rho_i)$ is the largest real root of $B_i(x)$ when $\tht_{i}^{2}>\thb^{2}$.
Evaluating~\eqref{eq:psip} in the limit $z=b^+$ yields
\begin{equation} \label{eq:Apsibroot}
A\{ \psi \left( b^{+}\right) \}=\frac{2cb}{\psi ^{\prime }\left(
b^{+}\right) }=0,
\end{equation}
because $\psi ^{\prime }\left( b^{+}\right) =\infty $ by Property~3 of Section~\ref{sct:props}.
If $\tht_{i}^{2}>\thb^2$ then $\rho_{i} = D^{-1}( 1/\tht
_{i}^{2}) $ and so
\begin{equation} \label{eq:Bpsirhoroot}
0=1-\tht_{i}^{2}D\left( \rho_{i}\right) =1-c\tht_{i}^{2}\sum_{\ell
=1}^{L}\frac{p_{\ell }}{\psi \left( \rho_{i}\right) -\sigma_{\ell }^{2}}
=B_{i}\{ \psi \left( \rho_{i}\right) \}.
\end{equation}
\eqref{eq:Apsibroot} shows that $\psi \left(b^{+}\right) $ is a real root of $A(x)$, and~\eqref{eq:Bpsirhoroot} shows that
 $\psi \left( \rho_{i}\right) $ is a real root of $B_{i}(x)$.

Recall that $\psi(b^+),\psi(\rho_i) \geq \max_\ell ( \sigma_\ell^2) $ by Property~2 of Section~\ref{sct:props}, and note that both $A(x)$ and $B_i(x)$ monotonically increase for $x>\max_\ell ( \sigma_\ell^2) $.
Thus each has exactly one real root larger than $\max_\ell ( \sigma_\ell^2)$, i.e., its largest real root, and so $\psi(b^+)=\alpha$ and $\psi(\rho_i)=\beta_i$ when $\tht_{i}^2>\thb^2$, where $\alpha$ and $\beta_i$ are the largest real roots of $A(x)$ and $B_i(x)$, respectively.

A subtle point is that $A(x)$ and $B_i(x)$ always have largest real roots $\alpha$ and $\beta$ even though $\psi(\rho_i)$ is defined only when $\tht_{i}^{2}>\thb^{2}$.
Furthermore, $\alpha$ and $\beta$ are always larger than $\max_\ell (\sigma_\ell^2)$ and both $A(x)$ and $B_i(x)$ are monotonically increasing in this regime and so we have the equivalence
\begin{equation} \label{eq:equiv}
B_{i}\left( \alpha \right) <0 \quad   \Leftrightarrow \quad   \alpha <\beta_{i} \quad  \Leftrightarrow   \quad 0<A\left( \beta
_{i}\right).
\end{equation}
Writing~\eqref{eq:cond_brho},~\eqref{eq:theta_brho} and~\eqref{eq:rs_brho}
in terms of $\alpha$ and $\beta_i$, then applying the equivalence~\eqref{eq:equiv} and combining with~\eqref{eq:initzero} yields the main results~\eqref{eq:rstheta},~\eqref{eq:rs} and~\eqref{eq:rsright}.

\section{Proof of Theorem~\ref{thm:bound}} \label{sct:bound}

If $A(\beta_i) \geq 0$ then~\eqref{eq:rs} and~\eqref{eq:rsright} increase with $A(\beta_i)$ and decrease with $\beta_i$ and $B'(\beta_i)$.
Similarly,~\eqref{eq:rstheta} increases with~$\beta_i$, as illustrated by~\eqref{eq:rstheta_alt}. As a result, Theorem~\ref{thm:bound} follows immediately from
the following bounds, all of which are met with equality if and only if $\sigma_1^2 = \cdots = \sigma_L^2$:
\begin{align}
\beta_i \geq c\tht_i^2+\sigmab^2 , \quad
B_i'(\beta_i) \geq \frac{1}{c\tht_i^2}  ,\quad
A(\beta_i) \leq 1-\frac{1}{c}\left(\frac{\sigmab}{\tht_i}\right)^4.
\label{eq:intbounds}
\end{align}
The bounds~\eqref{eq:intbounds} are shown by exploiting convexity to appropriately bound the rational functions $B_i(x)$, $B_i'(x)$ and $A(x)$.
We bound $\beta_i$ by noting that
\begin{equation*}
0 = B_i(\beta_i)
= 1- c\tht_i^2 \sum_{\ell=1}^L\frac{p_\ell}{\beta_i-\sigma_\ell^2}
\leq 1-\frac{c\tht_i^2}{\beta_i-\sigmab^2},
\end{equation*}
because $\sigma_\ell^2 < \beta_i$ and $f(v)= 1/(\beta_i - v)$ is a strictly convex function over $v < \beta_i$. Thus $\beta_i \geq c\tht_i^2 + \sigmab^2$.
We bound $B_i'(\beta_i)$ by noting that
\begin{equation*}
B_i'(\beta_i) = c\tht_i^2 \sum_{\ell=1}^L \frac{p_\ell}{(\beta_i-\sigma_\ell^2)^2}
\geq c\tht_i^2 \left(\sum_{\ell=1}^L \frac{p_\ell}{\beta_i-\sigma_\ell^2}\right)^2
= c\tht_i^2 \left(\frac{1}{c\tht_i^2}\right)^2
= \frac{1}{c\tht_i^2},
\end{equation*}
because the quadratic function $z^2$ is strictly convex.
Similarly,
\begin{equation*}
A(\beta_i) = 1 - c\sum_{\ell=1}^L \frac{p_\ell\sigma_\ell^4}{(\beta_i-\sigma_\ell^2)^2}
\leq 1 - c\left(\sum_{\ell=1}^L \frac{p_\ell\sigma_\ell^2}{\beta_i-\sigma_\ell^2}\right)^2
\leq 1-\frac{1}{c}\left(\frac{\sigmab}{\tht_i}\right)^4,
\end{equation*}
because the quadratic function $z^2$ is strictly convex and
\begin{equation*}
\sum_{\ell=1}^L \frac{p_\ell\sigma_\ell^2}{\beta_i-\sigma_\ell^2}
= \beta_i\sum_{\ell=1}^L \frac{p_\ell}{\beta_i-\sigma_\ell^2} - 1
= \frac{\beta_i}{c\tht_i^2} - 1
\geq \frac{c\tht_i^2+\sigmab^2}{c\tht_i^2} - 1
= \frac{\sigmab^2}{c\tht_i^2}.
\end{equation*}
All of the above bounds are met with equality if and only if $\sigma_1^2 = \cdots =\sigma_L^2$ because the convexity in all cases is strict.
As a result, homoscedastic noise minimizes~\eqref{eq:rstheta}, and it maximizes~\eqref{eq:rs} and~\eqref{eq:rsright}.
See Section~\ref{sct:add_props} of the Online Supplement for some interesting additional properties in this context.

\section{Discussion and extensions}

\label{sct:discussion}
This paper provided simplified expressions (Theorem~\ref{thm:rs}) for the asymptotic recovery of a low-dimensional subspace, the corresponding subspace amplitudes and the corresponding coefficients by the principal components, PCA amplitudes and scores, respectively, obtained from applying PCA to noisy high-dimensional heteroscedastic data.
The simplified expressions provide generalizations of previous results for the special case of homoscedastic data. They were derived by first connecting several recent results from random matrix theory~\cite{bai2010sao,benaych2012tsv} to obtain initial expressions for asymptotic recovery that are difficult to evaluate and analyze, and then exploiting a nontrivial structure in the expressions to find the much simpler algebraic descriptions of Theorem~\ref{thm:rs}.

These descriptions enable both easy and efficient calculation as well as reasoning about the asymptotic performance of PCA.
In particular, we use the simplified expressions to show that, for a fixed average noise variance, asymptotic subspace recovery, amplitude recovery and coefficient recovery are all worse when the noise is heteroscedastic as opposed to homoscedastic (Theorem~\ref{thm:bound}).
Hence, while average noise variance is often a practically convenient measure for the overall quality of data, it gives an overly optimistic estimate of PCA performance. Our expressions~\eqref{eq:rstheta},~\eqref{eq:rs} and~\eqref{eq:rsright} in Theorem~\ref{thm:rs} are more accurate.

We also investigated examples to gain insight into how the asymptotic performance of PCA depends on the model parameters:
sample-to-dimension ratio~$c$, subspace amplitudes~$\tht_1,\dots,\tht_k$, proportions~$p_1,\dots,p_L$ and noise variances~$\sigma_1^2,\dots,\sigma_L^2$.
We found that performance depends in expected ways on
\begin{itemize}
\item [a)]
sample-to-dimension ratio: performance improves with more samples;
\item [b)]
subspace amplitudes: performance improves with larger amplitudes;
\item [c)]
proportions: performance improves when more samples have low noise.
\end{itemize}
We also learned that when the gap between the two largest noise variances is ``sufficiently wide'', the performance is dominated by the largest noise variance.
This result provides insight into why PCA performs poorly in the presence of gross errors and why heteroscedasticity degrades performance in the sense of Theorem~\ref{thm:bound}.
Nevertheless, adding ``slightly'' noisier samples to an existing dataset can still improve PCA performance; even adding significantly noisier samples can be beneficial if they are sufficiently numerous.

Finally, we presented numerical simulations that demonstrated concentration of subspace recovery to the asymptotic prediction~\eqref{eq:rs} with good agreement for practical problem sizes.
The same agreement occurs for the PCA amplitudes and coefficient recovery.
The simulations also showed good agreement with the conjectured phase transition (Conjecture~\ref{conj:rs}).

There are many exciting avenues for extensions and further work.
An area of ongoing work is the extension of our analysis to a weighted version of PCA, where
the samples are first weighted to
reduce the impact of noisier points.
Such a method may be natural when the noise variances are known or can be estimated well.
Data formed in this way do not match the model of~\cite{benaych2012tsv}, and so the analysis involves first extending the results of~\cite{benaych2012tsv} to handle this more general case. Preliminary findings suggest that whitening the noise with inverse noise variance weights $1/\sigma_\ell^2$ is not optimal.

Another natural direction is to consider a general distribution of noise variances $\nu$, where we suppose that the empirical noise distribution
$(\delta_{\eta_1^2} + \cdots + \delta_{\eta_n^2}) /n \asto  \nu$
as $n \to \infty$.
We conjecture that if $\eta_1,\dots,\eta_n$ are bounded for all $n$ and $\int d\nu(\tau)/(x-\tau) \to \infty$ as $x \to \tau_\text{max}^+$,  then the almost sure limits in this paper hold but with
\begin{align*}
A\left( x\right) =1-c\int\frac{\tau^2d\nu(\tau)
}{\left( x-\tau\right) ^{2}},\quad
B_{i}\left( x\right) =1-c\tht_{i}^{2}\int\frac{d\nu(\tau)
}{x-\tau},
\end{align*}
where $\tau_\text{max}$ is the supremum of the support of $\nu$.
The proofs of Theorem~\ref{thm:rs} and Theorem~\ref{thm:bound} both generalize straightforwardly for the most part; the main trickiness comes in carefully arguing that limits pass through integrals in Section~\ref{sct:props}.

Proving that there is indeed a phase transition in the asymptotic subspace recovery and coefficient recovery,
as conjectured in Conjecture~\ref{conj:rs}, is another area of future work.
That proof may be of greater interest in the context of a weighted PCA method.
Another area of future work is explaining the puzzling phenomenon described in Section~\ref{sct:impact_sigma}, where, in some regimes, performance improves by increasing the noise variance.
More detailed analysis of the general impacts of the model parameters could also be interesting.
A final direction of future work is deriving finite sample results for heteroscedastic noise as was done for homoscedastic noise in~\cite{nadler2008fsa}.

\begin{appendix}
\section*{Acknowledgments}
The authors thank Raj Rao Nadakuditi and Raj Tejas Suryaprakash for many helpful discussions regarding the singular values and vectors of low rank perturbations of large random matrices.
The authors also thank Edgar Dobriban for his feedback on a draft and for pointing them to the generalized spiked covariance model.
The authors also thank Rina Foygel Barber for suggesting average inverse noise variance be considered and for pointing out that Theorem~\ref{thm:bound} implies an analogous statement for this measure.
Finally, the authors thank the editors and referees for their many helpful comments and suggestions that significantly improved the strength, clarity and style of the paper.
\end{appendix}

\begin{appendix}

\end{appendix}

\clearpage
\counterwithout{equation}{section}
\renewcommand{\thepage}{S\arabicminus{page}{23}}
\renewcommand{\thesection}{S\arabicminus{section}{0}}
\renewcommand{\thetable}{S\arabic{table}}
\renewcommand{\theequation}{S\arabic{equation}}
\renewcommand{\thefigure}{S\arabicminus{figure}{5}}

\pdfbookmark[0]{Supplement}{bookmark:supp}

\begin{frontmatter}
\title{Supplementary material for ``Asymptotic performance of PCA for\\ high-dimensional heteroscedastic Data''}

\begin{abstract}
This supplement fleshes out several details in~\cite{hong2017aposupp}.
Section~\ref{sct:spike} relates the model~\eqref{eq:model} in the paper to spiked covariance models~\cite{bai2012osesupp,johnstone2001otdsupp}. Section~\ref{sct:add_props} discusses interesting properties of the simplified expressions. Section~\ref{sct:analysis_right} shows the impact of the parameters on the asymptotic PCA amplitudes and coefficient recovery. Section~\ref{sct:exp_rest} contains numerical simulation results for PCA amplitudes and coefficient recovery, and Section~\ref{sct:add_exps} simulates complex-valued and Gaussian mixture data.
\end{abstract}
\end{frontmatter}

\section{Relationship to spiked covariance models} \label{sct:spike}
The model~\eqref{eq:model} considered in the paper is similar in spirit to the generalized spiked covariance model of~\cite{bai2012osesupp}.
To discuss the relationship more easily, we will refer to the paper model~\eqref{eq:model} as the ``inter-sample heteroscedastic model''.
Both this and the generalized spiked covariance model generalize the Johnstone spiked covariance model proposed in~\cite{johnstone2001otdsupp}. In the Johnstone spiked covariance model~\cite{bai2012osesupp}, sample vectors $y_1,\dots,y_n \in \bC^d$ are generated as
\begin{equation}
y_i =
\diag(\alpha_1^2,\dots,\alpha_k^2,\underbrace{1,\dots,1}_{d-k \text{ copies}})^{1/2} x_i, \label{eq:spike}
\end{equation}
where $x_i \in \bC^d$ are independent identically distributed (iid) vectors with iid entries that have mean $\bE (x_{ij}) = 0$ and variance $\bE |x_{ij}|^2=1$.

\begin{sloppypar}
For normally distributed subspace coefficients and noise vectors, the inter-sample heteroscedastic model~\eqref{eq:model} is equivalent (up to rotation) to generating sample vectors $y_1,\dots,y_n \in \bC^d$ as
\end{sloppypar}
\begin{equation} \label{eq:model-cov}
y_i =
\diag(\tht_1^2 + \eta_i^2,\dots,\tht_k^2 + \eta_i^2,\underbrace{\eta_i^2,\dots,\eta_i^2}_{d-k \text{ copies}})^{1/2} x_i,
\end{equation}
where $x_i \in \bC^d$ are iid with iid normally distributed entries.
\eqref{eq:model-cov} generalizes the Johnstone spiked covariance model because the covariance matrix can vary across samples.
Heterogeneity here is {\em across} samples; all entries $(y_i)_1,\dots,(y_i)_d$ within each sample $y_i$ have equal noise variance $\eta_i^2$.

The generalized spiked covariance model generalizes the Johnstone spiked covariance model differently. In the generalized spiked covariance model~\cite{bai2012osesupp}, sample vectors $y_1,\dots,y_n \in \bC^d$ are generated as
\begin{equation} \label{eq:gspike}
y_i =
\begin{bmatrix}
\LLambda &  \\
& \VV_{d-k}
\end{bmatrix}^{1/2}
x_i,
\end{equation}
where $x_i\in\bC^d$ are iid with iid entries as in~\eqref{eq:spike}, $\LLambda \in \bC^{k\times k}$ is a deterministic Hermitian matrix with eigenvalues $\alpha_1^2,\dots,\alpha_k^2$,  $\VV_{d-k} \in \bR^{(d-k) \times (d-k)}$ has limiting eigenvalue distribution $\nu$,
and these all satisfy a few technical conditions~\cite{bai2012osesupp}.
All samples share a common covariance matrix, but the model allows, among other things, for heterogenous variance within the samples. To illustrate this flexibility, note that we could set
\begin{align} \label{eq:gspikeh}
\LLambda &=
\diag(\tht_1^2 + \eta_1^2,\dots,\tht_k^2 + \eta_k^2),
 &
\VV_{d-k} &=
\diag(\eta_{k+1}^2,\dots,\eta_d^2).
\end{align}
In this case, there is heteroscedasticity among the entries of each sample vector.
Heterogeneity here is {\em within} each sample, not across them; recall that all samples have the same covariance matrix.

Therefore, for data with {\em intra}-sample heteroscedasticity, one should use the results of~\cite{bai2012osesupp} and~\cite{yao2015lscsupp} for the generalized spiked covariance model.
For data with {\em inter}-sample heteroscedasticity, one should use the new results presented in Theorem~\ref{thm:rs} of the paper~\cite{hong2017aposupp} for the inter-sample heteroscedastic model.
A couple variants of the inter-sample heteroscedastic model are also natural to consider in the context of spiked covariance models; the next two subsections discuss these.

\subsection{Random noise variances} \label{sct:randommodel}

\begin{figure}[t]
\centering
\begin{subfigure}[t]{0.31\linewidth}
\centering
\includegraphics[width=\linewidth]{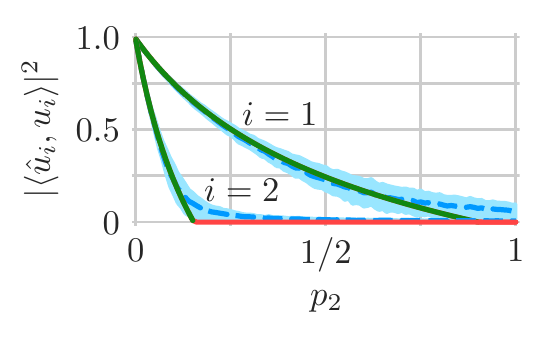}
\caption{Deterministic inter-sample heteroscedastic model.}
\label{fig:mixture_det}
\end{subfigure}
\quad
\begin{subfigure}[t]{0.31\linewidth}
\centering
\includegraphics[width=\linewidth]{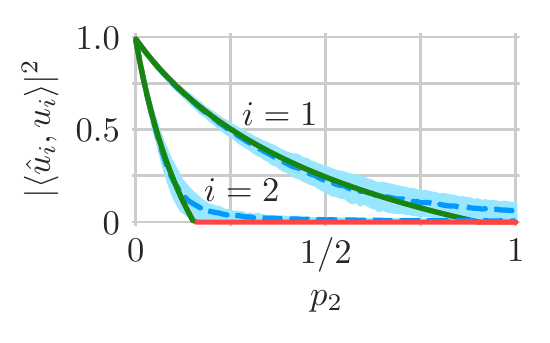}
\caption{Random inter-sample heteroscedastic model.}
\label{fig:mixture_rand}
\end{subfigure}
\quad
\begin{subfigure}[t]{0.31\linewidth}
\centering
\includegraphics[width=\linewidth]{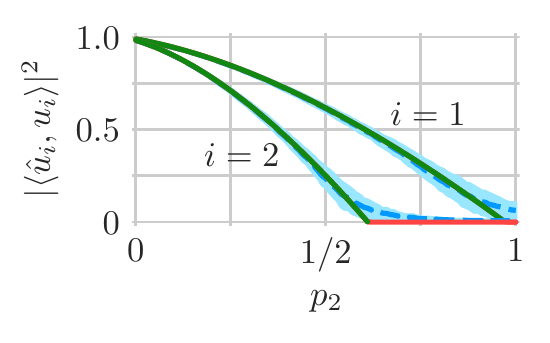}
\caption{Johnstone spiked covariance model with covariance~\eqref{eq:randcov}.}
\label{fig:mixture_spike}
\end{subfigure}
\caption{Simulated subspace recovery as a function of the contamination fraction $p_2$, the proportion of samples with noise variance $\sigma_2^2 = 3.25$, where the other noise variance~$\sigma_1^2=0.1$ occurs in proportion $p_1 = 1-p_2$.
Subspace amplitudes are $\tht_1=1$ and $\tht_2=0.8$, and there are $10^4$ samples in $10^3$ dimensions.
Simulation mean (dashed blue curve) and interquartile interval
(light blue ribbon) are shown with the asymptotic recovery~\eqref{eq:rs} of Theorem~\ref{thm:rs} (green curve).
The region where $A(\beta_i) \leq0$ is the red horizontal segment with  value zero (the prediction of Conjecture~\ref{conj:rs}).
Deterministic noise variances $\eta_1^2,\dots,\eta_n^2$ are used for simulations in (a), random ones are used for those in (b), and (c) has data generated according to the Johnstone spiked covariance model with covariance matrix set as~\eqref{eq:randcov}.
}
\label{fig:mixture}
\end{figure}

The noise variances $\eta_1^2,\dots,\eta_n^2$ in the inter-sample heteroscedastic model~\eqref{eq:model} are deterministic.
A natural variation could be to instead make them iid random variables defined as
\begin{equation} \label{eq:mixture}
\eta_i^2 =
\begin{cases}
\sigma_1^2 & \text{ with probability } p_1, \\
\; \vdots &\\
\sigma_L^2 & \text{ with probability } p_L, \\
\end{cases}
\end{equation}
where $p_1+\dots+p_L=1$.
To ease discussion, this section will
use the words ``deterministic'' and ``random'' before ``inter-sample heteroscedastic model'' to differentiate between the paper model~\eqref{eq:model} that has deterministic noise variances and its variant that instead has iid random noise variances~\eqref{eq:mixture}.
In the random inter-sample heteroscedastic model, scaled noise vectors $\eta_1\varepsilon_1,\dots,\eta_n\varepsilon_n$ are iid vectors drawn from a mixture. As a result, sample vectors $y_1,\dots,y_n$ are also iid vectors with covariance matrix (up to rotation)
\begin{equation} \label{eq:randcov}
\bE (y_iy_i^\CT) =
\diag(\tht_1^2 + \sigmab^2,\dots,\tht_k^2 + \sigmab^2,\underbrace{\sigmab^2,\dots,\sigmab^2}_{d-k \text{ copies}}),
\end{equation}
where $\sigmab^2 = p_1\sigma_1^2+\cdots+p_L\sigma_L^2$ is the average variance.

\eqref{eq:randcov} is a spiked covariance matrix and the samples $y_1,\dots,y_n$ are iid vectors, and so it could be tempting to think that the data can be equivalently generated from the Johnstone spiked covariance model with covariance matrix~\eqref{eq:randcov}.
However this is not true.
The PCA performance of the random inter-sample heteroscedastic model is similar to that of the deterministic version and is different from that of the Johnstone spiked covariance model with covariance matrix~\eqref{eq:randcov}.
Figure~\ref{fig:mixture} illustrates the distinction in numerical simulations.
In all simulations, we drew $10^4$ samples from a $10^3$ dimensional ambient space, where the subspace amplitudes were $\tht_1=1$ and $\tht_2=0.8$. Two noise variances $\sigma_1^2=0.1$ and $\sigma_2^2 = 3.25$ have proportions $p_1=1-p_2$ and $p_2$.
In Figure~\ref{fig:mixture_det}, data are generated according to the deterministic inter-sample heteroscedastic model.
In Figure~\ref{fig:mixture_rand}, data are generated according to the random inter-sample heteroscedastic model.
In Figure~\ref{fig:mixture_spike}, data are generated according to the Johnstone spiked covariance model with covariance matrix~\eqref{eq:randcov}.

Figures~\ref{fig:mixture_det} and~\ref{fig:mixture_rand} demonstrate that data generated according to the inter-sample heteroscedastic model have similar behavior whether the noise variances $\eta_1^2,\dots,\eta_n^2$ are set deterministically or randomly as~\eqref{eq:mixture}.
The similarity is expected because the random noise variances in the limit will equal $\sigma_1^2,\dots,\sigma_L^2$ in proportions approaching $p_1,\dots,p_L$ by the law of large numbers.
Thus data generated with random noise variances should have similar asymptotic PCA performance as data generated with deterministic noise variances.

Figures~\ref{fig:mixture_rand} and~\ref{fig:mixture_spike} demonstrate that data generated according to the random inter-sample heteroscedastic model behave quite differently from data generated according to the Johnstone spiked covariance model, even though both have iid sample vectors with covariance matrix~\eqref{eq:randcov}.
To understand why, recall that in the random inter-sample heteroscedastic model, the noise standard deviation $\eta_i$ is shared among the entries of the scaled noise vector $\eta_i\varepsilon_i$.
This induces statistical dependence among the entries of the sample vector $y_i$ that is not eliminated by whitening with $\bE (y_iy_i^\CT)^{-1/2}$.
Whitening a sample vector $y_i$ generated according to the Johnstone spiked covariance model, on the other hand, produces the vector $x_i$ that has iid entries by definition.
Thus, the random inter-sample heteroscedastic model is not equivalent to the Johnstone spiked covariance model.
One should use Theorem~\ref{thm:rs} in the paper~\cite{hong2017aposupp} to analyze asymptotic PCA performance in this setting rather than existing results for the Johnstone spiked covariance model~\cite{benaych2012tsvsupp,biehl1994smosupp,johnstone2009ocasupp,nadler2008fsasupp,paul2007aossupp}.

\subsection{Row samples}
In matrix form, the inter-sample heteroscedastic model can be written as
\begin{equation*}
\Y = (y_1, \dots, y_n)=\Ut\Tht \Zt^{\CT }+\E\Eta \in \bC^{d\times n},
\end{equation*}
where
\begin{itemize}
\item[] $\Zt= (\zt^{(1)}, \dots, \zt^{(k)}) \in \bC^{n\times k}$ is the coefficient matrix,
\item[] $\E= (\varepsilon_1,\dots,\varepsilon_n) \in \bC^{d\times n}$ is the (unscaled) noise matrix,
\item[] $\Eta = \diag(\eta_{1},\dots,\eta_{n}) \in \bR_+^{n\times n}$ is a diagonal matrix of noise standard deviations.
\end{itemize}
Samples in the paper~\cite{hong2017aposupp} are the columns $y_1,\dots,y_n$ of the data matrix $\Y$, but one could alternatively form samples from the rows
\begin{equation} \label{eq:rows}
y^{(i)} =
\begin{bmatrix}
(y_1)_i \\ \vdots \\ (y_n)_i
\end{bmatrix}
=\Zt^* \Tht \ut^{(i)} + \Eta\varepsilon^{(i)},
\end{equation}
where $\ut^{(i)} = ((\ut_1)_i, \dots, (\ut_n)_i)$ and $\varepsilon^{(i)} =((\varepsilon_1)_i, \dots, (\varepsilon_n)_i)$ are the $i$th rows of $\Ut$ and $\E$, respectively.
Row samples~\eqref{eq:rows} are exactly the columns of the transposed data matrix $\Y^\TT$ and so row samples have the same PCA amplitudes as column samples; principal components and score vectors swap.

In~\eqref{eq:rows}, noise heteroscedasticity is within each row sample $y^{(i)}$ rather than across row samples $y^{(1)},\dots,y^{(d)}$, and so one might think that the row samples could be equivalently generated from the generalized spiked covariance model~\eqref{eq:gspike} with a covariance similar to~\eqref{eq:gspikeh}.
However, the row samples are neither independent nor identically distributed; $\Ut$ induces dependence across rows as well as variety in their distributions. As a result, the row samples do not match the generalized spiked covariance model.

One could make~$\Ut$ random according to the ``i.i.d. model'' of~\cite{benaych2012tsvsupp}. As noted in Remark~\ref{rm:unitaryinv}, Theorem~\ref{thm:rs} from the paper~\cite{hong2017aposupp} still holds and the asymptotic PCA performance is unchanged.
For such~$\Ut$, the row samples $y^{(1)},\dots,y^{(d)}$ are now identically distributed but they are still not independent; dependence arises because~$\Zt$ is shared.
To remove the dependence, one could make $\Zt$ deterministic and also design it so that the row samples are iid with covariance matrix matching that of~\eqref{eq:gspike}, but doing so no longer matches the inter-sample heteroscedastic model.
It corresponds instead to having deterministic coefficients associated with a random subspace. Thus to analyze asymptotic PCA performance for row samples one should still use Theorem~\ref{thm:rs} in the paper~\cite{hong2017aposupp} rather than existing results for the generalized spiked covariance model~\cite{bai2012osesupp,yao2015lscsupp}.

\section{Additional properties} \label{sct:add_props}
This section highlights a few additional properties of $\beta_i$, $B'_i(\beta_i)$ and $A(\beta_i)$ that lend deeper insight into how they vary with the noise variances $\sigma_1^2,\dots,\sigma_L^2$.

\subsection{Expressing \texorpdfstring{$A(\beta_i)$}{A(beta\_i)} in terms of \texorpdfstring{$\beta_i$}{beta\_i} and \texorpdfstring{$B'_i(\beta_i)$}{Bi'(beta\_i)}}
We can rewrite $A(\beta_i)$ in terms of $\beta_i$ and $B'_i(\beta_i)$ as follows:
\begin{align}
A\left( \beta_{i}\right)  &=1-c\sum_{\ell =1}^{L}\frac{p_{\ell }\sigma
_{\ell }^{4}}{\left( \beta_{i}-\sigma_{\ell }^{2}\right) ^{2}}=1-c\sum_{\ell =1}^{L}p_{\ell }\left\{ 1-\frac{-2\beta_{i}\sigma_{\ell
}^{2}+\beta_{i}^{2}}{\left( \beta_{i}-\sigma_{\ell }^{2}\right) ^{2}}\right\}  \nonumber \\&
=1-c\sum_{\ell =1}^{L}p_{\ell } \left\{ 1-\frac{-2\beta_{i}\sigma_{\ell
}^{2}+2\beta_{i}^{2}-\beta_{i}^{2}}{\left( \beta_{i}-\sigma_{\ell
}^{2}\right) ^{2}}\right\}  \nonumber\\
&=1-c\sum_{\ell =1}^{L}p_{\ell } \left\{ 1+\beta_{i}^{2}\frac{1}{\left(
\beta_{i}-\sigma_{\ell }^{2}\right) ^{2}}-2\beta_{i}\frac{1}{\beta
_{i}-\sigma_{\ell }^{2}}\right\}  \nonumber\\&
=1-c\sum_{\ell =1}^{L}p_{\ell }-c\beta_{i}^{2}\sum_{\ell =1}^{L}\frac{p_{\ell }}{\left( \beta_{i}-\sigma_{\ell }^{2}\right) ^{2}}+2c\beta
_{i}\sum_{\ell =1}^{L}\frac{p_{\ell }}{\beta_{i}-\sigma_{\ell }^{2}} \nonumber\\
&=1-c-c\beta_{i}^{2} \left\{ \frac{1}{c\tht_{i}^{2}}B_{i}^{\prime }\left(
\beta_{i}\right) \right\} +2c\beta_{i} \left\{ \frac{1-B_{i}\left( \beta
_{i}\right) }{c\tht_{i}^{2}}\right\}  \nonumber\\&
=1-c-\frac{\beta_i}{\tht_{i}^{2}} \{ \beta_{i}B_{i}^{\prime }\left(
\beta_{i}\right) -2\}, \label{eq:AintermsbetaB}
\end{align}
since $B_i(\beta_i)=0$.
Thus we focus on properties of $\beta_i$ and $B'_i(\beta_i)$ for the remainder of Section~\ref{sct:add_props}; \eqref{eq:AintermsbetaB} relates them back to $A(\beta_i)$.

\subsection{Graphical illustration of \texorpdfstring{$\beta_i$}{beta\_i}}

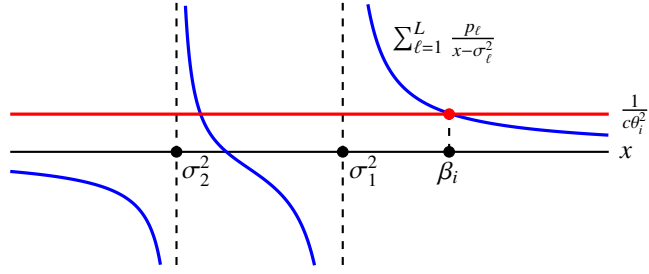
\begin{figure}[t]
\centering \vspace{-7mm}
\begin{tikzpicture}[xscale=1.75,yscale=0.5,domain=-1:5,samples=400]
  \draw[thick] (-0.5,0) -- (4,0) node[right] {$x$};
  \draw[thick,dashed] (0.75,-3) -- (0.75,3.95);
  \filldraw (0.75,0) circle [x radius=0.4mm, y radius=1.4mm];
  \draw (0.9,0.15) node[below] {$\sigma_2^2$};
  \draw[thick,dashed] (2,-3) -- (2,3.95);
  \filldraw (2,0) circle [x radius=0.4mm, y radius=1.4mm];
  \draw (2.16,0.15) node[below] {$\sigma_1^2$};

  \draw[blue,very thick,domain=-0.5:0.631]  plot (\x,{0.7/(\x-2)+0.3/(\x-0.75)});
  \draw[blue,very thick,domain=0.8172:1.79] plot (\x,{0.7/(\x-2)+0.3/(\x-0.75)});
  \draw[blue,very thick,domain=2.189:4]    plot (\x,{0.7/(\x-2)+0.3/(\x-0.75)});
  \draw (2.3,2.9) node[right] {$\sum_{\ell=1}^L \frac{p_\ell}{x-\sigma_\ell^2}$};

  \draw[red,very thick] (-0.5,1) -- (4,1) node[black,right] {$\frac{1}{c\tht_i^2}$};
  \draw[thick,dashed] (2.8,0) -- (2.8,1);
  \filldraw[red] (2.8,1) circle [x radius=0.4mm, y radius=1.4mm];
  \filldraw (2.8,0) circle [x radius=0.4mm, y radius=1.4mm] node[below] {$\beta_i$};
\end{tikzpicture}
\caption{Location of the largest real root $\beta_i$ of $B_i(x)$ for two noise variances $\sigma_1^2 = 2$ and $\sigma_2^2 =0.75$, occurring in proportions~$p_1=70\%$ and $p_2 = 30\%$, where the sample-to-dimension ratio is $c=1$ and the subspace amplitude is $\tht_i=1$.}
\label{fig:betaRoot}
\end{figure}

Note that $\beta_i$ is the largest solution of
\begin{equation} \label{eq:betaRoot}
\frac{1}{c\tht_i ^{2}}=\sum_{\ell =1}^{L}\frac{p_{\ell }}{x-\sigma_{\ell
}^{2}},
\end{equation}
because $\beta_i$ is the largest real root of $B_i(x)$.
Figure~\ref{fig:betaRoot} illustrates~\eqref{eq:betaRoot} for two noise variances $\sigma_1^2 = 2$ and $\sigma_2^2 =0.75$, occurring in proportions~$p_1=70\%$ and $p_2 = 30\%$, where the sample-to-dimension ratio is $c=1$ and the subspace amplitude is $\tht_i=1$.
The plot is a graphical representation of $\beta_i$ and gives a way to visualize the relationship between $\beta_i$ and the model parameters. Observe, for example, that $\beta_i$ is larger than all the noise variances and that increasing $\tht_i$ or $c$ amounts to moving the horizontal red line down and tracking the location of the intersection.

\subsection{Level curves}

\begin{figure}[t]
\centering
\begin{subfigure}[t]{0.48\linewidth}
\centering
\includegraphics[scale=0.9]{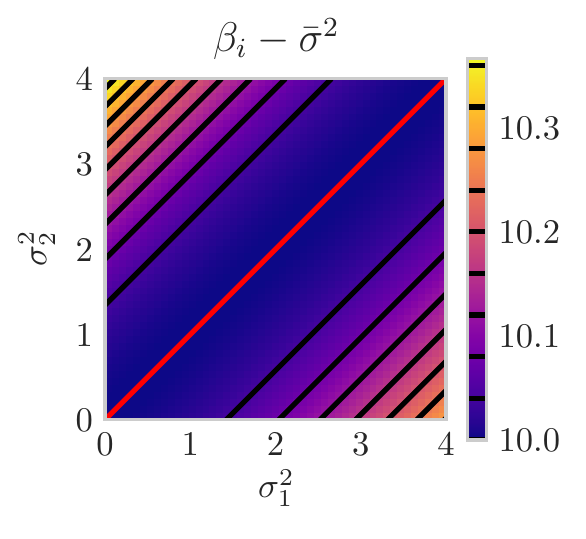}
\caption{$\beta_i - \sigmab^2$ over $\sigma_1^2$ and $\sigma_2^2$.}
\end{subfigure}
\quad
\begin{subfigure}[t]{0.48\linewidth}
\centering
\includegraphics[scale=0.9]{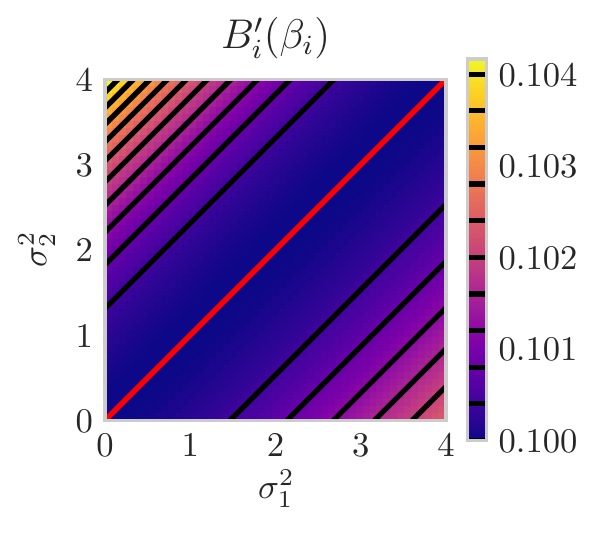}
\caption{$B'_i(\beta_i)$ over $\sigma_1^2$ and $\sigma_2^2$.}
\end{subfigure}
\caption{Illustration of $\beta_i - \sigmab^2$ and $B'_i(\beta_i)$ as a function of two noise variances $\sigma_1^2$ and $\sigma_2^2$.
The level curves are along lines parallel to $\sigma_1^2 = \sigma_2^2$ for all values of sample-to-dimension ratio $c$, proportions $p_1$ and $p_2$, and subspace amplitude $\tht_i$}
\label{fig:betaBp}
\end{figure}

Figure~\ref{fig:betaBp} shows $\beta_i - \sigmab^2$ and $B'_i(\beta_i)$ as functions (implicitly) of $L=2$  noise variances $\sigma_1^2$ and $\sigma_2^2$, where
\begin{equation*}
\sigmab^2 = p_1\sigma_1^2 + \cdots + p_L\sigma_L^2
\end{equation*}
is the average noise variance.
Figure~\ref{fig:betaBp} illustrates that lines parallel to the diagonal $\sigma_1^2=\sigma_2^2$ are level curves for both $\beta_i - \sigmab^2$ and $B'_i(\beta_i)$.
This is a general phenomenon: lines parallel to the diagonal $\sigma_1^2=\cdots=\sigma_L^2$ are level curves of both $\beta_i -\sigmab^2$ and $B'_i(\beta_i)$ for all sample-to-dimension ratios $c$, proportions $p_1,\dots,p_L$ and subspace amplitudes $\tht_i$.

To show this fact, note that $\beta_i-\sigmab^2$ is the largest real solution to
\begin{equation} \label{eq:betabar}
0=B_i(x+\sigmab^2)
= 1-c\tht_i^2 \sum_{\ell=1}^L \frac{p_\ell}{x-(\sigma_\ell^2-\sigmab^2)},
\end{equation}
because $0=B_i(\beta_i)$. Changing the noise variances to $\sigma_1^2+\Delta,\dots,\sigma_L^2+\Delta$ for some $\Delta$ also changes the average noise variance to $\sigmab^2 + \Delta$ and so $\sigma_\ell^2-\sigmab^2$ remains unchanged. As a result, the solutions to~\eqref{eq:betabar} remain unchanged.

Similarly, note that
\begin{equation} \label{eq:Bpbar}
B'_i(\beta_i) = c\tht_i^2 \sum_{\ell=1}^L \frac{p_\ell}{(\beta_i-\sigma_\ell^2)^2}
= c\tht_i^2 \sum_{\ell=1}^L \frac{p_\ell}{ \{(\beta_i-\sigmab^2)-(\sigma_\ell^2-\sigmab^2)\}^2 }
\end{equation}
remains unchanged when changing the noise variances to $\sigma_1^2+\Delta,\dots,\sigma_L^2+\Delta$.

Thus we conclude from~\eqref{eq:betabar} and~\eqref{eq:Bpbar} that lines parallel to $\sigma_1^2=\cdots=\sigma_L^2$ are level curves for both $\beta_i -\sigmab^2$ and $B'_i(\beta_i)$.
The line $\sigma_1^2=\cdots=\sigma_L^2$ in particular minimizes the value of both, as was established in the proof of Theorem~\ref{thm:bound}.

\subsection{Hessians along the line \texorpdfstring{$\sigma_1^2=\cdots=\sigma_L^2$}{sigma1\^{}2=...=sigmaL\^{}2}}
We consider $\beta_i-\sigmab^2$ and $B'_i(\beta_i)$ as functions (implicitly) of the noise variances $\sigma_1^2,\dots,\sigma_L^2$.
To denote derivatives more clearly, we  denote the $i$th noise variance as $v_i = \sigma_i^2$.

Written in this notation, we have
\begin{align}
0 = 1-c\tht_i^2 \sum_{\ell=1}^L \frac{p_\ell}{\beta_i-v_\ell}, \label{eq:betav} \\
B'_i(\beta_i) = c\tht_i^2 \sum_{\ell=1}^L \frac{p_\ell}{(\beta_i-v_\ell)^2}. \label{eq:Bpv}
\end{align}
Taking the total derivative of~\eqref{eq:betav} with respect to $v_s$ and $v_t$ and solving for $\partial^2\beta_i/(\partial v_t \partial v_s)$ yields an initially complicated expression, but evaluating it on the line $v_1=\cdots=v_L$ vastly simplifies it, yielding:
\begin{equation} \label{eq:hessbeta}
\frac{\partial^2(\beta_i-\sigmab^2)}{\partial v_t\partial v_s}
=\frac{2}{c\tht_i^2}(p_s \delta_{s,t} - p_s p_t).
\end{equation}
where $\delta_{s,t} = 1$ if $s=t$ and $0$ otherwise.
Notably, $\sigmab^2 = p_1v_1+\cdots+p_Lv_L$ has zero Hessian everywhere.

Likewise, taking the total derivative of~\eqref{eq:Bpv} with
respect to $v_s$ and $v_t$ yields an initially complicated expression that is again vastly simplified by evaluating it on the line $v_1=\cdots=v_L$, yielding:
\begin{equation} \label{eq:hessBp}
\frac{\partial^2 B'_i(\beta_i)}{\partial v_t \partial v_s}
=\frac{2}{(c\tht_i^2)^4}(p_s \delta_{s,t} - p_s p_t).
\end{equation}
\eqref{eq:hessbeta} and~\eqref{eq:hessBp} show that the Hessian matrices for $\beta_i-\sigmab^2$ and $B'_i(\beta_i)$ are both scaled versions of the matrix
\begin{equation} \label{eq:hessbetaBp}
\Hess =
\underbrace{
\begin{bmatrix}
p_1 & & \\
& \ddots & \\
& & p_L
\end{bmatrix}
}_{\diag(p)}
-
\underbrace{
\begin{bmatrix}
p_1 \\ \vdots \\ p_L
\end{bmatrix}
\begin{bmatrix}
p_1 & \cdots & p_L
\end{bmatrix}
}_{pp^\TT}
\end{equation}
on the line $v_1=\cdots=v_L$.
The (scaled) Hessian matrix~\eqref{eq:hessbetaBp} is a rank one perturbation by $-pp^\TT$ of $\diag(p)$, and so its eigenvalues downward interlace with those of $\diag(p)$ (see Theorem~8.1.8 of~\cite{golub1996mcsupp}).
Namely, $\Hess$ has eigenvalues $\lambda_1,\dots,\lambda_L$ satisfying
\begin{equation*}
\lambda_1 \leq p_{(1)} \leq \lambda_2 \leq \cdots \leq \lambda_L \leq p_{(L)},
\end{equation*}
where $p_{(1)},\dots ,p_{(L)}$ are the proportions in increasing order.
The vector $\One$ of all ones, i.e., the vector in the direction of $v_1=\cdots=v_L$, is an eigenvector of~$\Hess$ with eigenvalue zero; note that
$\Hess\One=\diag(p) \One-pp^\TT\One=p-p=0$.
This eigenvalue is less than $p_{(1)}>0$ and so $\lambda_1=0$ and $\lambda_2,\dots,\lambda_L\geq p_{(1)}>0$.
Hence the Hessians of $\beta_i-\sigmab^2$ and $B'_i(\beta_i)$ are both zero in the direction of the line $v_1=\cdots=v_L$ and positive definite in other directions. This property provides deeper insight into the fact that $\beta_i-\sigmab^2$ and $B'_i(\beta_i)$ are minimized on the line $\sigma_1^2=\cdots=\sigma_L^2$, as was established in the proof of Theorem~\ref{thm:bound}.

\clearpage
\section{Impact of parameters: amplitude and coefficient recovery} \label{sct:analysis_right}

\begin{sloppypar}
Section~\ref{sct:analysis} of~\cite{hong2017aposupp} discusses how the asymptotic subspace recovery~\eqref{eq:rs} of Theorem~\ref{thm:rs} depends on the model parameters: sample-to-dimension ratio~$c$, subspace amplitudes~$\tht_1,\dots,\tht_k$, proportions~$p_1,\dots,p_L$ and noise variances~$\sigma_1^2,\dots,\sigma_L^2$.
This section shows that the same phenomena occur for the asymptotic PCA amplitudes~\eqref{eq:rstheta} and coefficient recovery~\eqref{eq:rsright}.
For the asymptotic PCA amplitudes, we consider the ratio $\thh_i^2/\tht_i^2$. As discussed in Remark~\ref{rm:bias}, the asymptotic PCA amplitude $\thh_i$ is positively biased relative to the subspace amplitude $\tht_i$, and so the almost sure limit of $\thh_i^2/\tht_i^2$ is greater than one, with larger values indicating more bias.
\end{sloppypar}

\subsection{Impact of sample-to-dimension ratio~\texorpdfstring{$c$}{c} and subspace amplitude~\texorpdfstring{$\tht_i$}{theta\_i}}

\begin{figure}[t]
\centering
\begin{subfigure}[t]{0.49\linewidth}
\centering
\includegraphics[width=\linewidth]{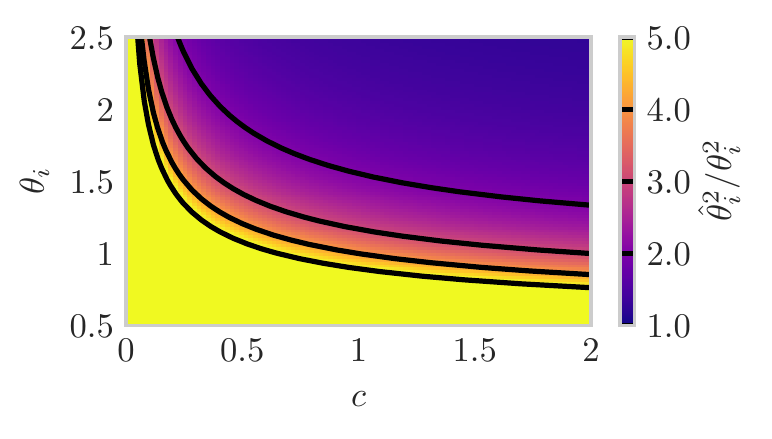}
\caption{Homoscedastic noise with $\sigma_1^2=1$.}
\label{fig:qual2a_t}
\end{subfigure}
\
\begin{subfigure}[t]{0.49\linewidth}
\centering
\includegraphics[width=\linewidth]{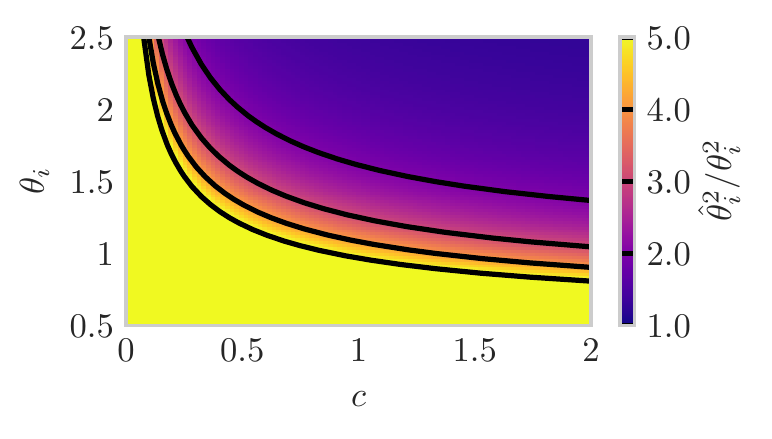}
\caption{Heteroscedastic noise with $p_1 = 80\%$ of samples at $\sigma_1^2=0.8$  and $p_2 = 20\%$ of samples at $\sigma_2^2=1.8$.}
\label{fig:qual2b_t}
\end{subfigure}
\caption{Asymptotic amplitude bias~\eqref{eq:rstheta} of the $i$th PCA amplitude as a function of sample-to-dimension ratio~$c$ and subspace amplitude~$\tht_i$ with average noise variance equal to one.
Contours are overlaid in black. The contours in~(b) are slightly further up and to the right than in (a);
more samples are needed to reduce the positive bias.}
\label{fig:qual2_t}
\end{figure}

\begin{figure}[t!]
\centering
\begin{subfigure}[t]{0.49\linewidth}
\centering
\includegraphics[width=\linewidth]{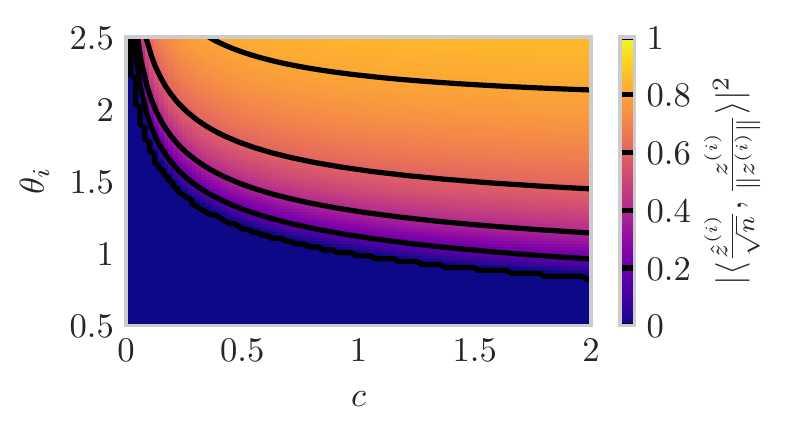}
\caption{Homoscedastic noise with $\sigma_1^2=1$.}
\label{fig:qual2a_z}
\end{subfigure}
\
\begin{subfigure}[t]{0.49\linewidth}
\centering
\includegraphics[width=\linewidth]{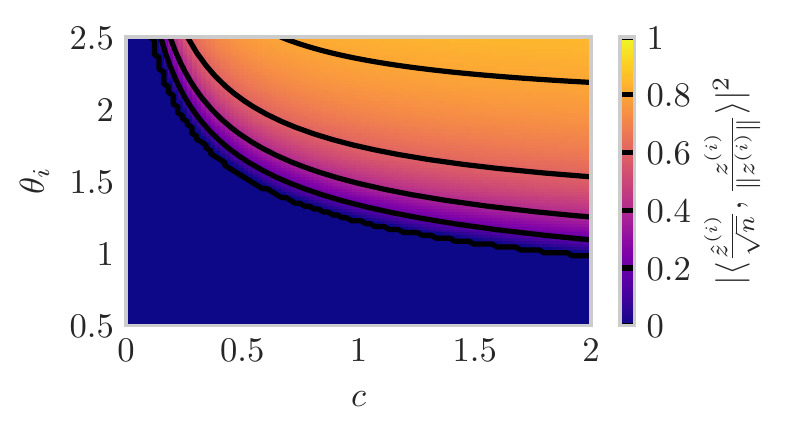}
\caption{Heteroscedastic noise with $p_1 = 80\%$ of samples at $\sigma_1^2=0.8$  and $p_2 = 20\%$ of samples at $\sigma_2^2=1.8$.}
\label{fig:qual2b_z}
\end{subfigure}
\caption{Asymptotic coefficient recovery~\eqref{eq:rsright} of the $i$th score vector as a function of sample-to-dimension ratio~$c$ and subspace amplitude~$\tht_i$ with average noise variance equal to one.
Contours are overlaid in black and the region where $A(\beta_i) \leq0$ is shown as zero (the prediction of Conjecture~\ref{conj:rs}). The phase transition in~(b) is further right than in (a);
more samples are needed to recover the same strength signal.}
\label{fig:qual2_z}
\end{figure}

As in Section~\ref{sct:analysis_ctheta}, we vary the sample-to-dimension ratio $c$ and subspace amplitude $\tht_i$ in two scenarios:
\begin{enumerate}
\item[a)] there is only one noise variance fixed at $\sigma_1^2=1$
\item[b)] there are two noise variances~$\sigma_1^2=0.8$ and~$\sigma_2^2=1.8$ occurring in proportions~$p_1 = 80\%$ and $p_2 = 20\%$.
\end{enumerate}
Both scenarios have average noise variance $1$.
Figures~\ref{fig:qual2_t} and~\ref{fig:qual2_z} show analogous plots to Figure~\ref{fig:qual2} but for the asymptotic PCA amplitudes~\eqref{eq:rstheta} and coefficient recovery~\eqref{eq:rsright}, respectively.

As was the case for Figure~\ref{fig:qual2} in Section~\ref{sct:analysis_ctheta}, decreasing the subspace amplitude $\tht_i$ degrades both the asymptotic amplitude performance (i.e., increases bias) shown in Figure~\ref{fig:qual2_t} and the asymptotic coefficient recovery shown in Figure~\ref{fig:qual2_z}, but the lost performance could be regained by increasing the number of samples.
Furthermore, both the asymptotic amplitude performance shown in Figure~\ref{fig:qual2_t} and the asymptotic coefficient recovery shown in Figure~\ref{fig:qual2_z} decline when the noise is heteroscedastic.
Though the difference is subtle for the asymptotic amplitude bias, the contours move up and to the right in both cases.
This degradation is consistent with Theorem~\ref{thm:bound}; PCA performs worse on heteroscedastic data than it does on homoscedastic data of the same average noise variance and more samples or a larger subspace amplitude are needed to compensate.

\subsection{Impact of proportions~\texorpdfstring{$p_1,\dots,p_L$}{p1,...,pL}}

\begin{figure}[t]
\centering
\begin{subfigure}[t]{0.495\linewidth}
\centering
\includegraphics[height=3.95cm]{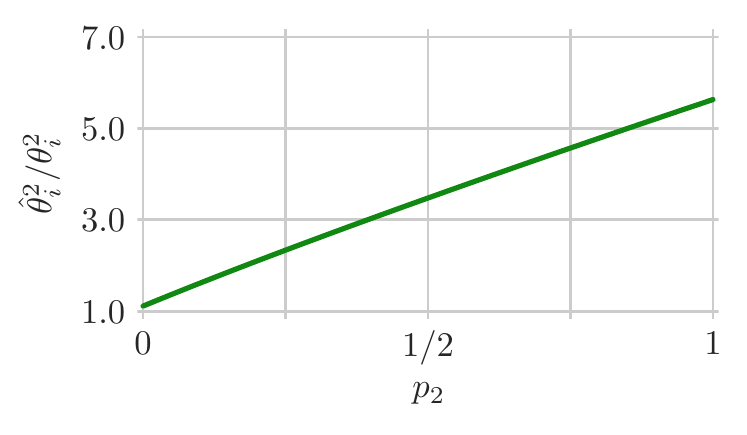}
\caption{Asymptotic amplitude bias~\eqref{eq:rstheta}.}
\label{fig:qual3_t}
\end{subfigure}
\begin{subfigure}[t]{0.495\linewidth}
\centering
\includegraphics[height=3.95cm]{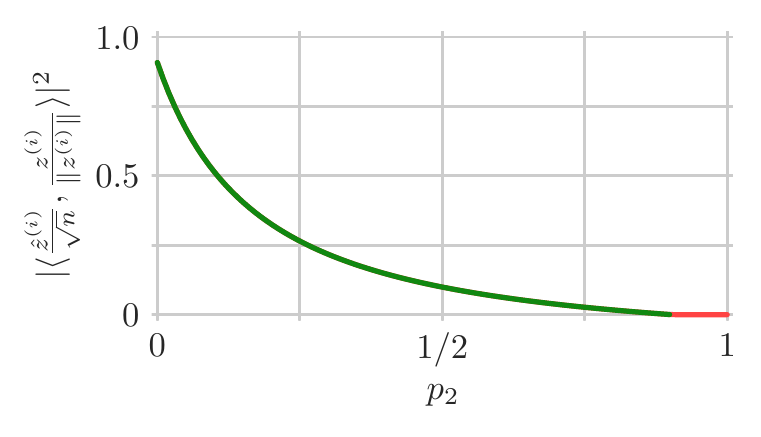}
\caption{Asymptotic coefficient recovery~\eqref{eq:rsright}.}
\label{fig:qual3_z}
\end{subfigure}
\caption{Asymptotic amplitude bias~\eqref{eq:rstheta} and coefficient recovery~\eqref{eq:rsright} of the $i$th PCA amplitude and score vector as functions of the contamination fraction $p_2$, the proportion of samples with noise variance $\sigma_2^2 = 3.25$, where the other noise variance~$\sigma_1^2=0.1$ occurs in proportion $p_1 = 1-p_2$.
The sample-to-dimension ratio is $c=10$ and the subspace amplitude is $\tht_i=1$. The region where $A(\beta_i) \leq0$ is the red horizontal segment in (b) with  value zero (the prediction of Conjecture~\ref{conj:rs}).}
\label{fig:qual3_tz}
\end{figure}

As in Section~\ref{sct:qual3}, we consider two noise variances~$\sigma_1^2=0.1$ and~$\sigma_2^2=3.25$ occurring in proportions~$p_1=1-p_2$ and $p_2$, where
the sample-to-dimension ratio is $c=10$ and the subspace amplitude is~$\tht_i=1$.
Figure~\ref{fig:qual3_tz} shows analogous plots to Figure~\ref{fig:qual3} but for the asymptotic PCA amplitudes~\eqref{eq:rstheta} and coefficient recovery~\eqref{eq:rsright}.
As was the case for Figure~\ref{fig:qual3} in Section~\ref{sct:qual3}, performance generally degrades in Figure~\ref{fig:qual3_tz} as $p_2$ increases and low noise samples with noise variance $\sigma_1^2$ are traded for high noise samples with noise variance $\sigma_2^2$.
The performance is best when $p_2 = 0$ and all the samples have the smaller noise variance~$\sigma_1^2$, i.e., there is no contamination.

\subsection{Impact of noise variances~\texorpdfstring{$\sigma_1^2,\dots,\sigma_L^2$}{sigma1\^{}2,...,sigmaL\^{}2}}

\begin{figure}[t]
\centering
\begin{subfigure}[t]{0.495\linewidth}
\centering
\includegraphics[width=0.88\linewidth]{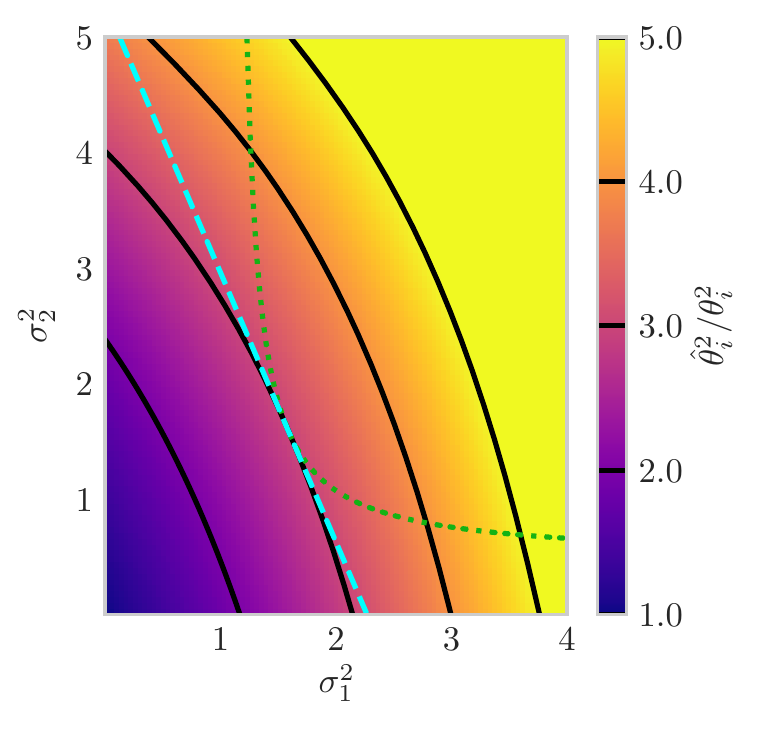}
\caption{Asymptotic amplitude bias~\eqref{eq:rstheta}.}
\label{fig:qual1_t}
\end{subfigure}
\begin{subfigure}[t]{0.495\linewidth}
\centering
\includegraphics[width=0.9\linewidth]{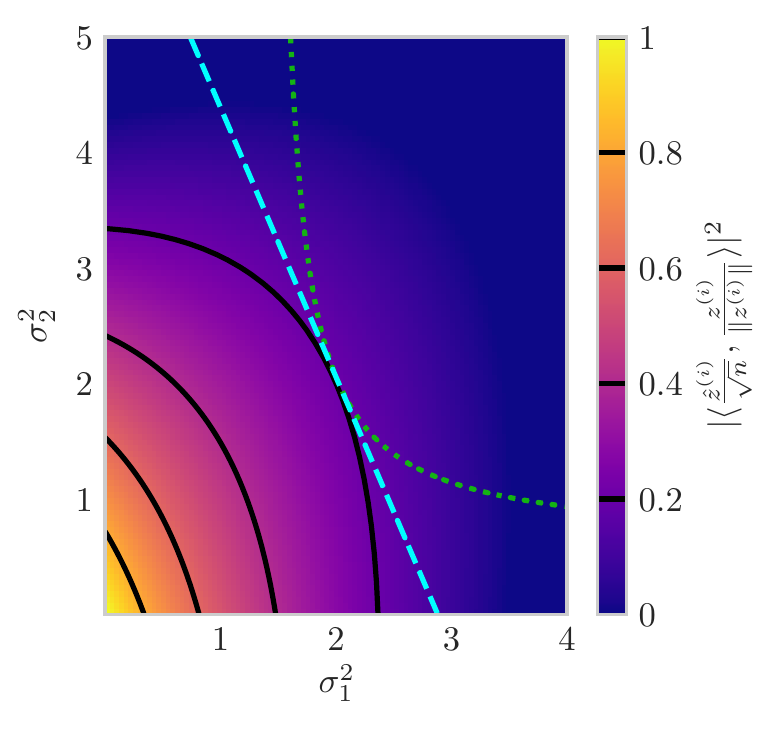}
\caption{Asymptotic coefficient recovery~\eqref{eq:rsright}.}
\label{fig:qual1_z}
\end{subfigure}
\caption{Asymptotic amplitude bias~\eqref{eq:rstheta} and coefficient recovery~\eqref{eq:rsright} of the $i$th PCA amplitude and score vector as functions of noise variances~$\sigma_1^2$ and~$\sigma_2^2$ occurring in proportions~$p_1 = 70\%$ and $p_2 = 30\%$, where
the sample-to-dimension ratio is $c=10$ and the subspace amplitude is $\tht_i=1$. Contours are overlaid in black and the region where $A(\beta_i) \leq0$ is shown as zero in (b), matching the prediction of Conjecture~\ref{conj:rs}.
Along each dashed cyan line, the average noise variance is fixed and the best performance occurs when $\sigma_1^2=\sigma_2^2=\sigmab^2$.
Along each dotted green curve, the average inverse noise variance is fixed and the best performance again occurs when $\sigma_1^2=\sigma_2^2$.
}
\label{fig:qual1_tz}
\end{figure}

As in Section~\ref{sct:impact_sigma}, we consider two noise variances~$\sigma_1^2$ and~$\sigma_2^2$ occurring in proportions~$p_1 = 70\%$ and $p_2 = 30\%$, where
the sample-to-dimension ratio is $c=10$ and the subspace amplitude is $\tht_i=1$.
Figure~\ref{fig:qual1_tz} shows analogous plots to Figure~\ref{fig:qual1} but for the asymptotic PCA amplitudes~\eqref{eq:rstheta} and coefficient recovery~\eqref{eq:rsright}.
As was the case for Figure~\ref{fig:qual1} in Section~\ref{sct:impact_sigma}, performance typically degrades with increasing noise variances.
The contours in Figure~\ref{fig:qual1_z} are also generally horizontal for small $\sigma_1^2$ and vertical for small $\sigma_2^2$. They indicate that when the gap between the two largest noise variances is ``sufficiently'' wide, the asymptotic coefficient recovery is roughly determined by the largest noise variance.
This property mirrors the asymptotic subspace recovery and occurs for similar reasons, discussed in detail in Section~\ref{sct:impact_sigma}.
Along each dashed cyan line in Figure~\ref{fig:qual1_tz}, the average noise variance is fixed and the best performance for both the PCA amplitudes and coefficient recovery again occurs when $\sigma_1^2=\sigma_2^2=\sigmab^2$,
as was  predicted by Theorem~\ref{thm:bound}.
Along each dotted green curve in Figure~\ref{fig:qual1_tz}, the average inverse noise variance is fixed and the best performance for both the PCA amplitudes and coefficient recovery again occurs when $\sigma_1^2=\sigma_2^2$, as was predicted in Remark~\ref{rem:bound}.

\subsection{Impact of adding data} \label{sct:add_tz}
\begin{figure}[t!]
\centering
\begin{subfigure}[t]{0.495\linewidth}
\centering
\includegraphics[width=0.923\linewidth]{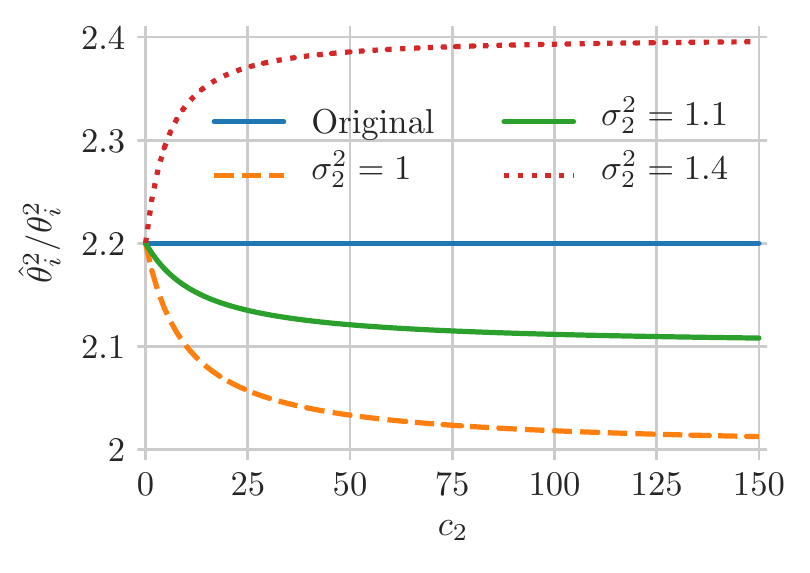}
\caption{Asymptotic amplitude bias~\eqref{eq:rstheta}.}
\label{fig:impact_add_t}
\end{subfigure}
\begin{subfigure}[t]{0.495\linewidth}
\centering
\includegraphics[width=0.95\linewidth]{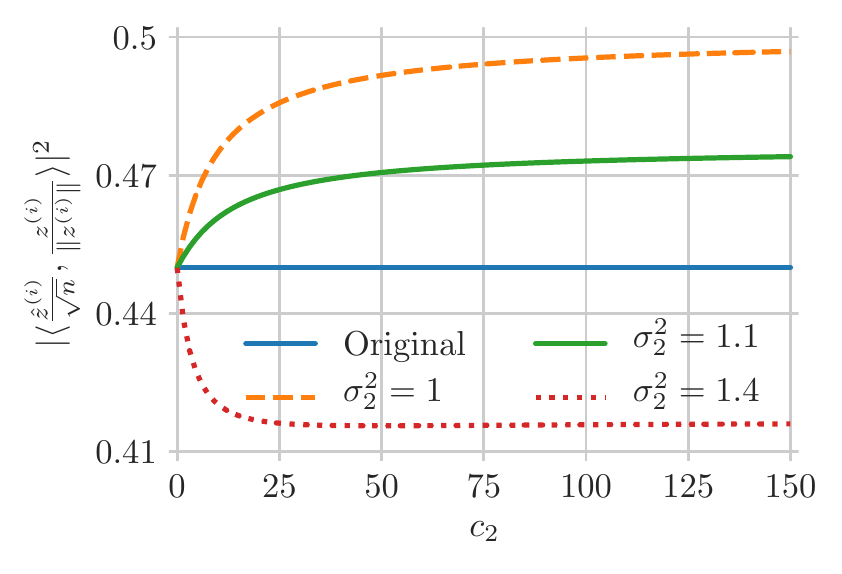}
\caption{Asymptotic coefficient recovery~\eqref{eq:rsright}.}
\label{fig:impact_add_z}
\end{subfigure}
\caption{Asymptotic amplitude bias~\eqref{eq:rstheta} and coefficient recovery~\eqref{eq:rsright} of the $i$th PCA amplitude and score vector for samples added with noise variance $\sigma_2^2$ and samples-per-dimension $c_2$ to an existing dataset with noise variance~$\sigma_1^2 = 1$, sample-to-dimension ratio $c_1=10$ and subspace amplitude $\tht_i=1$.}
\label{fig:impact_add_tz}
\end{figure}

As in Section~\ref{sct:impact_add}, we consider adding data with noise variance~$\sigma_2^2$ and sample-to-dimension ratio~$c_2$ to an existing dataset that has noise variance~$\sigma_1^2 = 1$, sample-to-dimension ratio $c_1=10$ and subspace amplitude $\tht_i=1$  for the $i$th component.
The combined dataset has a sample-to-dimension ratio of $c=c_1+c_2$ and is potentially heteroscedastic with noise variances $\sigma_1^2$ and $\sigma_2^2$ appearing in proportions $p_1 = c_1/c$ and $p_2=c_2/c$.

Figure~\ref{fig:impact_add_tz} shows analogous plots to Figure~\ref{fig:impact_add} in Section~\ref{sct:impact_add} but for the asymptotic PCA amplitudes~\eqref{eq:rstheta} and coefficient recovery~\eqref{eq:rsright}.
As was the case for Figure~\ref{fig:impact_add}, the dashed orange curves show the recovery when $\sigma_2^2 = 1 = \sigma_1^2$ and illustrate the benefit we would expect for homoscedastic data: increasing the samples per dimension improves recovery.
The green curves show the performance when $\sigma_2^2=1.1>\sigma_1^2$; as before, these samples are ``slightly'' noisier and performance improves for any number added.
Finally, the dotted red curves show the performance when $\sigma_2^2=1.4>\sigma_1^2$.
As before, performance degrades when adding a small number of these noisier samples.
However, unlike subspace recovery, performance degrades when adding any amount of these samples.
In the limit $c_2 \to \infty$, the asymptotic amplitude bias is $1+\sigma_2^2/\theta_i^2$ and the asymptotic coefficient recovery is $1/(1+\sigma_2^2/\theta_i^2)$;
neither has perfect recovery in the limit when added samples are noisy.

\section{Numerical simulation: amplitude and coefficient recovery} \label{sct:exp_rest}

\begin{figure}[t]
\centering
\begin{subfigure}[t]{0.49\linewidth}
\centering
\includegraphics[width=0.97\linewidth]{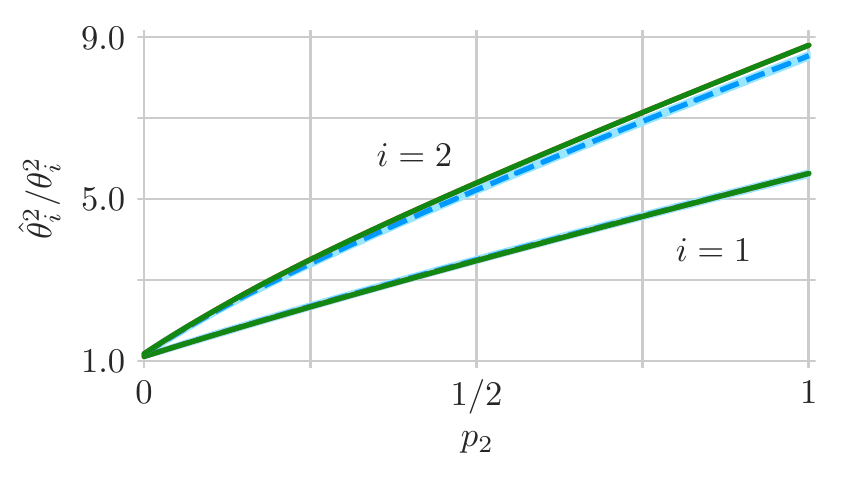}
\caption{$10^3$ samples in $10^2$ dimensions.}
\label{fig:exps51_t}
\end{subfigure}\
\begin{subfigure}[t]{0.49\linewidth}
\centering
\includegraphics[width=0.97\linewidth]{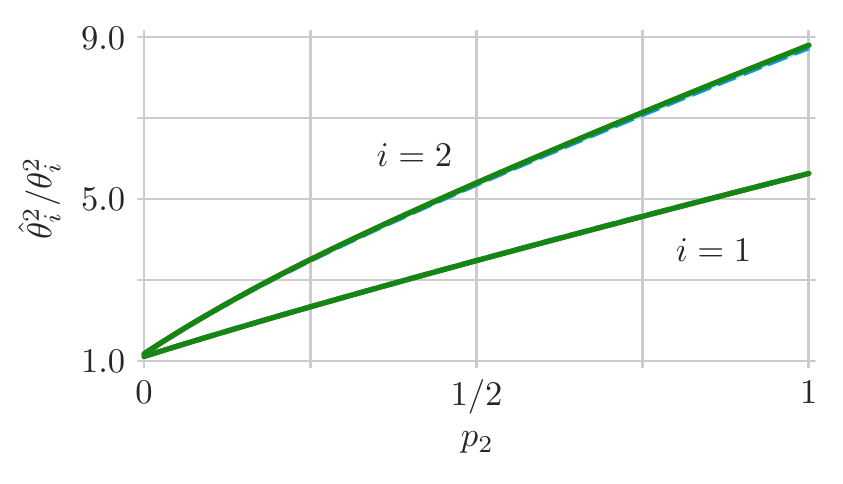}
\caption{$10^4$ samples in $10^3$ dimensions.}
\label{fig:exps52_t}
\end{subfigure}
\caption{Simulated amplitude bias~\eqref{eq:rstheta} as a function of the contamination fraction $p_2$, the proportion of samples with noise variance $\sigma_2^2 = 3.25$, where the other noise variance~$\sigma_1^2=0.1$ occurs in proportion $p_1 = 1-p_2$.
The sample-to-dimension ratio is $c=10$ and the subspace amplitudes are $\tht_1=1$ and $\tht_2=0.8$.
Simulation mean (dashed blue curve) and interquartile interval
(light blue ribbon) are shown with the asymptotic bias~\eqref{eq:rstheta} of Theorem~\ref{thm:rs} (green curve).
Increasing data size from (a) to (b) results in  even smaller interquartile intervals, indicating concentration to the mean, which is converging to the asymptotic bias.
}
\label{fig:exps5_t}
\end{figure}

\begin{figure}[t!]
\centering
\begin{subfigure}[t]{0.49\linewidth}
\centering
\includegraphics[width=0.97\linewidth]{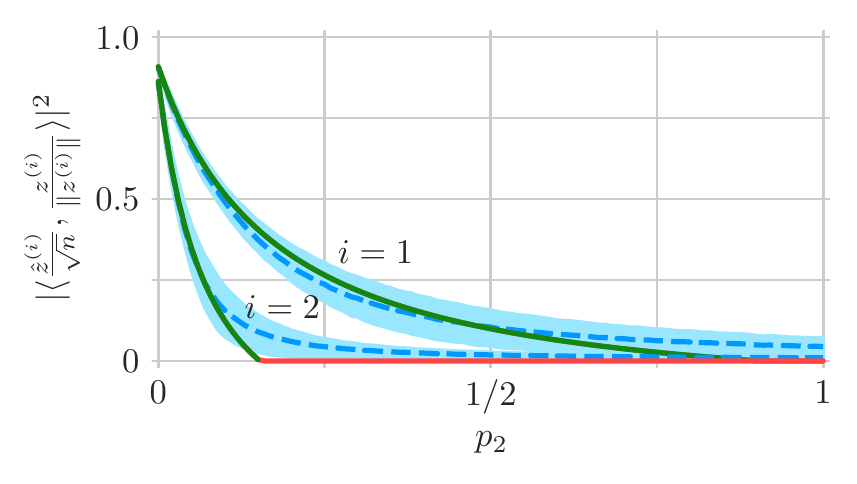}
\caption{$10^3$ samples in $10^2$ dimensions.}
\label{fig:exps51_z}
\end{subfigure}\
\begin{subfigure}[t]{0.49\linewidth}
\centering
\includegraphics[width=0.97\linewidth]{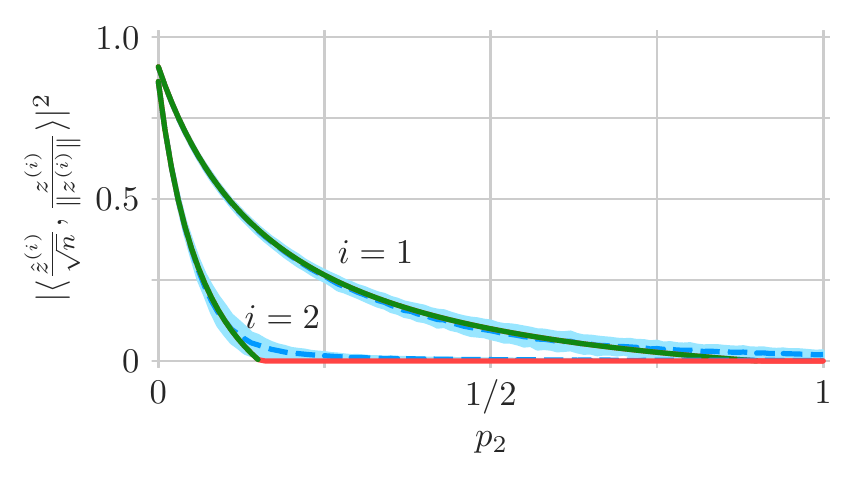}
\caption{$10^4$ samples in $10^3$ dimensions.}
\label{fig:exps52_z}
\end{subfigure}
\caption{Simulated coefficient recovery~\eqref{eq:rsright} as a function of the contamination fraction $p_2$, the proportion of samples with noise variance $\sigma_2^2 = 3.25$, where the other noise variance~$\sigma_1^2=0.1$ occurs in proportion $p_1 = 1-p_2$.
The sample-to-dimension ratio is $c=10$ and the subspace amplitudes are $\tht_1=1$ and $\tht_2=0.8$.
Simulation mean (dashed blue curve) and interquartile interval
(light blue ribbon) are shown with the asymptotic recovery~\eqref{eq:rsright} of Theorem~\ref{thm:rs} (green curve).
The region where $A(\beta_i) \leq0$ is the red horizontal segment with value zero (the prediction of Conjecture~\ref{conj:rs}).
Increasing data size from (a) to (b) results in  smaller interquartile intervals, indicating concentration to the mean, which is converging to the asymptotic recovery.
}
\label{fig:exps5_z}
\end{figure}

Section~\ref{sct:experiment} of~\cite{hong2017aposupp} shows that the asymptotic subspace recovery~\eqref{eq:rs} of Theorem~\ref{thm:rs} is meaningful for practical settings with finitely many samples in a finite-dimensional space.
This section shows that the same is true for the asymptotic PCA amplitudes~\eqref{eq:rstheta} and coefficient recovery~\eqref{eq:rsright}.
For the asymptotic PCA amplitudes, we again consider the ratio $\thh_i^2/\tht_i^2$. As discussed in Remark~\ref{rm:bias}, the asymptotic PCA amplitude $\thh_i$ is positively biased relative to the subspace amplitude $\tht_i$, and so the almost sure limit of $\thh_i^2/\tht_i^2$ is greater than one, with larger values indicating more bias.

As in Section~\ref{sct:experiment}, this section simulates data according to the model described in Section~\ref{sct:model} for a two-dimensional subspace with subspace amplitudes $\tht_1 = 1$ and $\tht_2 = 0.8$, two noise variances $\sigma_1^2=0.1$ and~$\sigma_2^2=3.25$, and a sample-to-dimension ratio of $c=10$. We sweep the proportion of high noise points $p_2$ from zero to one, setting $p_1=1-p_2$ as in Section~\ref{sct:experiment}.
The first simulation considers $n=10^{3}$ samples in a $d=10^{2}$ dimensional ambient space ($10^4$ trials). The second increases these to $n=10^{4}$ samples in a $d=10^{3}$ dimensional ambient space ($10^3$ trials).
All simulations generate data from the standard normal distribution, i.e., $\zt_{ij},\varepsilon_{ij}\sim\cN(0,1)$.
Figures~\ref{fig:exps5_t} and~\ref{fig:exps5_z} show analogous plots to Figure~\ref{fig:exps5} but for the asymptotic PCA amplitudes~\eqref{eq:rstheta} and coefficient recovery~\eqref{eq:rsright}, respectively.

As was the case for Figure~\ref{fig:exps5} in Section~\ref{sct:experiment}, both Figures~\ref{fig:exps5_t} and~\ref{fig:exps5_z} illustrate the following general observations:
\begin{enumerate}
\item[a)] the simulation mean and almost sure limit generally agree in the smaller simulation of $10^3$ samples in a $10^2$ dimensional ambient space
\item[b)] the smooth simulation mean deviates from the non-smooth almost sure limit near the phase transition
\item[c)] the simulation mean and almost sure limit agree better for the larger simulation of $10^4$ samples in a $10^3$ dimensional ambient space
\item[d)] the interquartile intervals for the larger simulations are roughly half the size of those in the smaller simulations, indicating concentration to the means.
\end{enumerate}
In fact, the amplitude bias in Figure~\ref{fig:exps5_t} and the coefficient recovery in Figure~\ref{fig:exps5_z} both have significantly better agreement with their almost sure limits than the subspace recovery in Figure~\ref{fig:exps5} has with its almost sure limit.
The amplitude bias in Figure~\ref{fig:exps5_t}, in particular, is tightly concentrated around its almost sure limit~\eqref{eq:rstheta}.
Furthermore, Figure~\ref{fig:exps5_z} demonstrates good agreement with Conjecture~\ref{conj:rs}, providing evidence that there is indeed a phase transition below which the coefficients are also not recovered.

\section{Additional numerical simulations} \label{sct:add_exps}
\begin{figure}[t]
\centering
\begin{subfigure}[t]{0.31\linewidth}
\centering
\includegraphics[width=\linewidth]{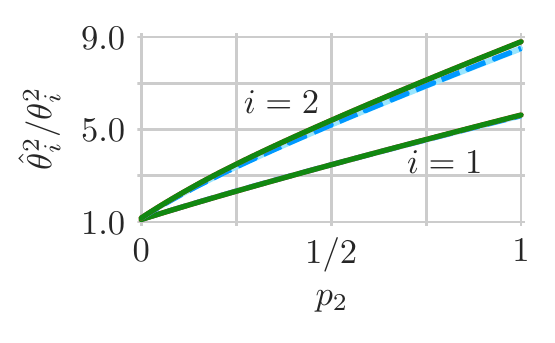}
\caption{Amplitude bias, $10^3$~samples in $10^2$ dimensions.}
\end{subfigure}
\quad
\begin{subfigure}[t]{0.31\linewidth}
\centering
\includegraphics[width=\linewidth]{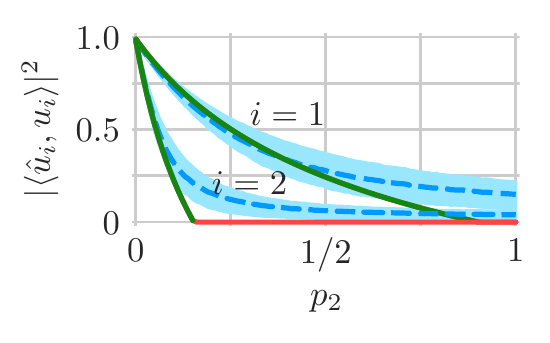}
\caption{Subspace recovery, $10^3$ samples in $10^2$ dimensions.}
\end{subfigure}
\quad
\begin{subfigure}[t]{0.31\linewidth}
\centering
\includegraphics[width=\linewidth]{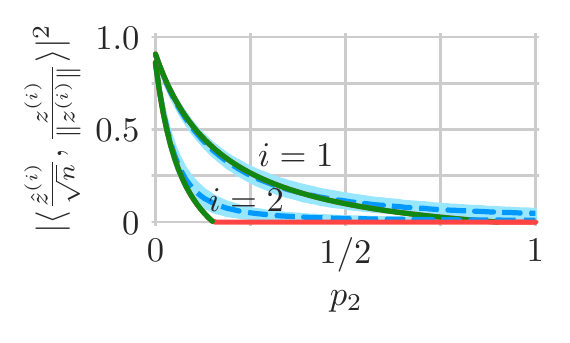}
\caption{Coefficient recovery, $10^3$ samples in $10^2$ dimensions.}
\end{subfigure}
\begin{subfigure}[t]{0.31\linewidth}
\centering
\includegraphics[width=\linewidth]{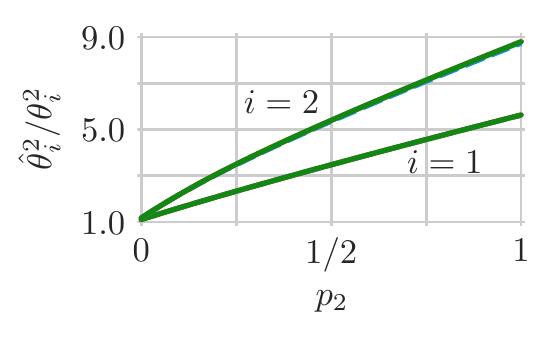}
\caption{Amplitude bias, $10^4$~samples in $10^3$ dimensions.}
\end{subfigure}
\quad
\begin{subfigure}[t]{0.31\linewidth}
\centering
\includegraphics[width=\linewidth]{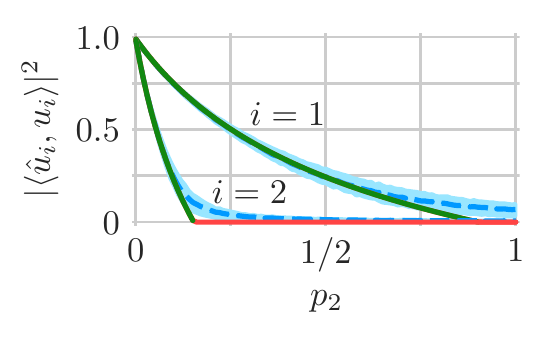}
\caption{Subspace recovery, $10^4$ samples in $10^3$ dimensions.}
\end{subfigure}
\quad
\begin{subfigure}[t]{0.31\linewidth}
\centering
\includegraphics[width=\linewidth]{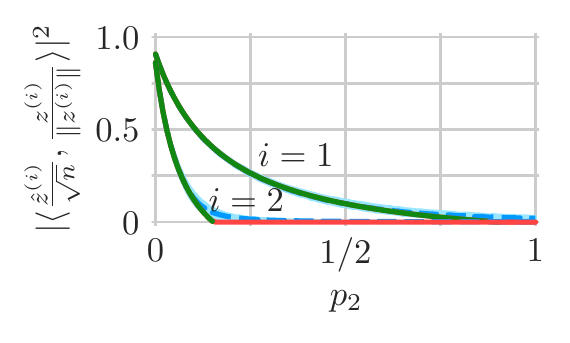}
\caption{Coefficient recovery, $10^4$ samples in $10^3$ dimensions.}
\end{subfigure}
\caption{Simulated complex-normal PCA performance as a function of the contamination fraction $p_2$, the proportion of samples with noise variance $\sigma_2^2 = 3.25$, where the other noise variance~$\sigma_1^2=0.1$ occurs in proportion $p_1 = 1-p_2$.
The sample-to-dimension ratio is $c=10$ and the subspace amplitudes are $\tht_1=1$ and $\tht_2=0.8$.
Simulation mean (dashed blue curve) and interquartile interval
(light blue ribbon) are shown with the almost sure limits of Theorem~\ref{thm:rs} (green curve).
The region where $A(\beta_i) \leq0$ is shown as red horizontal segments with value zero (the prediction of Conjecture~\ref{conj:rs}).
}
\label{fig:exp_complexnormal}
\end{figure}

\begin{figure}[t]
\centering
\begin{subfigure}[t]{0.31\linewidth}
\centering
\includegraphics[width=\linewidth]{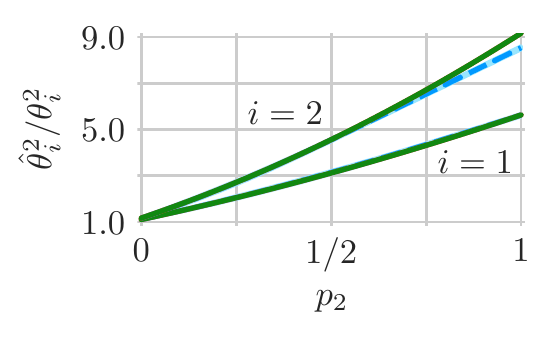}
\caption{Amplitude bias, $10^3$~samples in $10^2$ dimensions.}
\end{subfigure}
\quad
\begin{subfigure}[t]{0.31\linewidth}
\centering
\includegraphics[width=\linewidth]{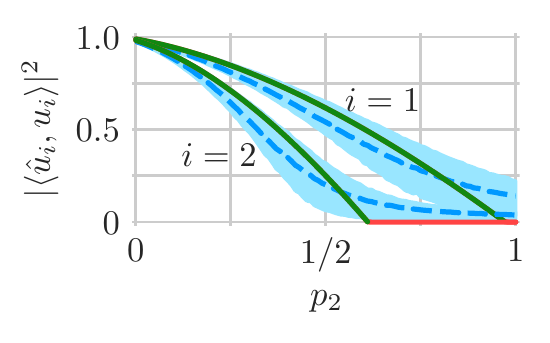}
\caption{Subspace recovery, $10^3$ samples in $10^2$ dimensions.}
\end{subfigure}
\quad
\begin{subfigure}[t]{0.31\linewidth}
\centering
\includegraphics[width=\linewidth]{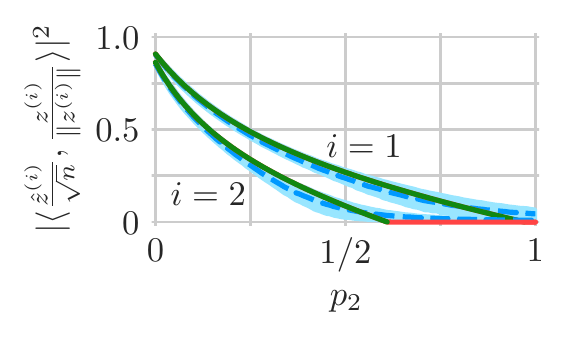}
\caption{Coefficient recovery, $10^3$ samples in $10^2$ dimensions.}
\end{subfigure}
\begin{subfigure}[t]{0.31\linewidth}
\centering
\includegraphics[width=\linewidth]{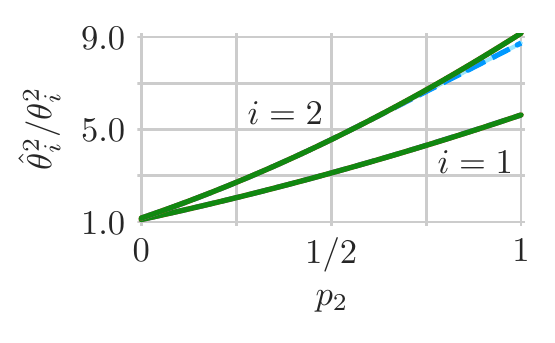}
\caption{Amplitude bias, $10^4$~samples in $10^3$ dimensions.}
\end{subfigure}
\quad
\begin{subfigure}[t]{0.31\linewidth}
\centering
\includegraphics[width=\linewidth]{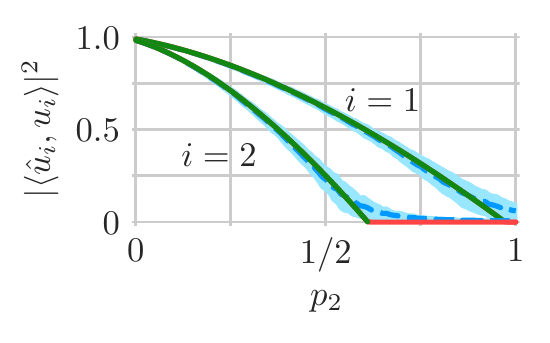}
\caption{Subspace recovery, $10^4$ samples in $10^3$ dimensions.}
\end{subfigure}
\quad
\begin{subfigure}[t]{0.31\linewidth}
\centering
\includegraphics[width=\linewidth]{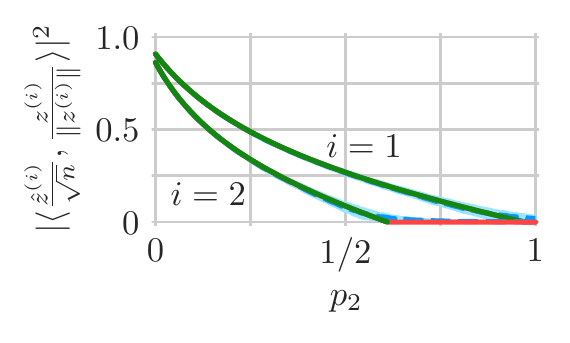}
\caption{Coefficient recovery, $10^4$ samples in $10^3$ dimensions.}
\end{subfigure}
\caption{Simulated mixture model PCA performance as a function of the mixture probability $p_2$, the probability that a scaled noise entry $\eta_i\varepsilon_{ij}$ is Gaussian with variance $\lambda_2^2 = 3.25$, where it is Gaussian with variance~$\lambda_1^2=0.1$ otherwise, i.e., with probability $p_1 = 1-p_2$.
The sample-to-dimension ratio is $c=10$ and the subspace amplitudes are $\tht_1=1$ and $\tht_2=0.8$.
Simulation mean (dashed blue curve) and interquartile interval
(light blue ribbon) are shown with the almost sure limits of Theorem~\ref{thm:rs} (green curve).
The region where $A(\beta_i) \leq0$ is shown as red horizontal segments with value zero (the prediction of Conjecture~\ref{conj:rs}).
}
\label{fig:exp_mixture}
\end{figure}

Section~\ref{sct:experiment} of~\cite{hong2017aposupp} and Section~\ref{sct:exp_rest} provide numerical simulation results for real-valued data generated using normal distributions. This section illustrates the generality of the model in Section~\ref{sct:model} by showing analogous simulation results for circularly symmetric complex normal data in Figure~\ref{fig:exp_complexnormal} and for a mixture of Gaussians in Figure~\ref{fig:exp_mixture}.
As before, we show the results of two simulations for each setting.
The first simulation considers $n=10^{3}$ samples in a $d=10^{2}$ dimensional ambient space ($10^4$ trials). The second increases these to $n=10^{4}$ samples in a $d=10^{3}$ dimensional ambient space ($10^3$ trials).

Figure~\ref{fig:exp_complexnormal} mirrors Sections~\ref{sct:experiment} and~\ref{sct:exp_rest} and simulates data according to the model described in Section~\ref{sct:model} for a two-dimensional subspace with subspace amplitudes $\tht_1 = 1$ and $\tht_2 = 0.8$, two noise variances $\sigma_1^2=0.1$ and~$\sigma_2^2=3.25$, and a sample-to-dimension ratio of $c=10$. We again sweep the proportion of high noise points $p_2$ from zero to one, setting $p_1=1-p_2$.
The only difference is that Figure~\ref{fig:exp_complexnormal} generates data from the standard {\em complex} normal distribution, i.e., $\zt_{ij},\varepsilon_{ij}\sim\cCN(0,1)$.

Figure~\ref{fig:exp_mixture} instead simulates a {\em homoscedastic} setting of the model described in Section~\ref{sct:model} over a range of noise distributions, all {\em mixtures} of Gaussians.
As before, we consider a two-dimensional subspace with subspace amplitudes $\tht_1 = 1$ and $\tht_2 = 0.8$, and a sample-to-dimension ratio of $c=10$.
Figure~\ref{fig:exp_mixture} generates coefficients $\zt_{ij}\sim\cN(0,1)$ from the standard normal distribution and generates noise entries $\varepsilon_{ij}$ from the Gaussian mixture model
\begin{equation*}
\varepsilon_{ij} \sim
\begin{cases}
\cN\left(0,\lambda_1^2/\sigma^2\right) & \text{with probability } p_1, \\
\cN\left(0,\lambda_2^2/\sigma^2\right) & \text{with probability } p_2,
\end{cases}
\end{equation*}
where $\lambda_1^2=0.1$ and~$\lambda_2^2=3.25$, and the {\em single} noise variance is set to
\begin{equation} \label{eq:mixture_var}
\sigma^2 = p_1\lambda_1^2+p_2\lambda_2^2.
\end{equation}
Each scaled noise entry $\eta_i\varepsilon_{ij}=\sigma\varepsilon_{ij}$ is a mixture of two Gaussian distributions with variances $\lambda_1^2$ and $\lambda_2^2$.
We sweep the mixture probability $p_2$ from zero to one, setting $p_1 = 1-p_2$.
Thus, Figure~\ref{fig:exp_mixture} illustrates performance over a range of noise distributions.
The noise variance~\eqref{eq:mixture_var} in Figure~\ref{fig:exp_mixture} matches the average noise variance in Figure~\ref{fig:exp_complexnormal} as we sweep $p_2$.
However, Figures~\ref{fig:exp_mixture} and~\ref{fig:exp_complexnormal} differ because Figure~\ref{fig:exp_mixture} simulates a {\em homoscedastic} setting while Figure~\ref{fig:exp_complexnormal} simulates a {\em heteroscedastic} setting.
Figure~\ref{fig:exp_mixture} also differs from Figure~\ref{fig:mixture_rand} that simulates data from the random inter-sample heteroscedastic model of Section~\ref{sct:randommodel}. While both simulate (scaled) noise from a mixture model, scaled noise entries $\eta_i\varepsilon_{ij}$ in Figure~\ref{fig:exp_mixture} are all iid. Scaled noise entries $\eta_i\varepsilon_{ij}$ in the random inter-sample heteroscedastic model are  independent only across samples; they are {\em not} independent within each sample. Figure~\ref{fig:exp_mixture} is instead more like Figure~\ref{fig:mixture_spike} that simulates data from the Johnstone spiked covariance model.
See Section~\ref{sct:randommodel} for a comparison of these models.

As was the case for (real-valued) standard normal data in Sections~\ref{sct:experiment} and~\ref{sct:exp_rest}, Figures~\ref{fig:exp_complexnormal} and~\ref{fig:exp_mixture} illustrate the following general observations:
\begin{enumerate}
\item[a)] the simulation means and almost sure limits generally agree in the smaller simulations of $10^3$ samples in a $10^2$ dimensional ambient space
\item[b)] the smooth simulation means deviate from the non-smooth almost sure limits near the phase transitions
\item[c)] the simulation means and almost sure limits agree better for the larger simulations of $10^4$ samples in a $10^3$ dimensional ambient space
\item[d)] the interquartile intervals for the larger simulations are roughly half the size of those in the smaller simulations, indicating concentration to the means.
\end{enumerate}
The agreement between simulations and almost sure limits demonstrated in both Figures~\ref{fig:exp_complexnormal} and~\ref{fig:exp_mixture} highlights the generality of the model considered in~\cite{hong2017aposupp}: it allows for both complex-valued data and non-Gaussian distributions.
In both cases, the asymptotic results of Theorem~\ref{thm:rs} remain meaningful for practical settings with finitely many samples in a finite-dimensional space.

\appendix


\end{document}